\documentclass{amsart}

\usepackage{amssymb,latexsym}
\usepackage{graphicx}
\usepackage[all,2cell]{xy}
\SelectTips{eu}{12}

\newcommand{\Z}{\mathbb{Z}}

\newcommand{\Q}{\mathbb{Q}}
\newcommand{\R}{\mathbb{R}}
\newcommand{\C}{\mathbb{C}}

\newtheorem{lemma}{Lemma}[section]

\newtheorem{proposition}[lemma]{Proposition}
\newtheorem{theorem}[lemma]{Theorem}
\newtheorem{corollary}[lemma]{Corollary}
\newtheorem{conjecture}[lemma]{Conjecture}

\theoremstyle{definition}
\newtheorem{remark}[lemma]{Remark}
\newtheorem{definition}[lemma]{Definition}

\begin{document}
\parindent0em
\setlength\parskip{.1cm}
\title{The isomorphism problem for planar polygon spaces}
\author{Dirk Sch\"utz}
\address{Department of Mathematical Sciences\\ University of Durham\\ Science Laboritories\\ South Rd\\ Durham DH1 3LE\\ United Kingdom}
\email{dirk.schuetz@durham.ac.uk}
\begin{abstract}
 We give a proof of a Conjecture of Walker which states that one can recover the lengths of the bars of a circular linkage from the cohomology ring of the configuration space. For a large class of length vectors, this has been shown by Farber, Hausmann and Sch\"utz. In the remaining cases, we use Morse theory and the fundamental group to describe a subring of the cohomology invariant under graded ring isomorphism. From this subring the conjecture can be derived by applying a result of Gubeladze on the isomorphism problem of monoidal rings.
\end{abstract}

\maketitle
\section{Introduction}

Let $\R^n_{>0}$ be the set of all $n$-tupels $\ell=(\ell_1,\ldots,\ell_n)$ with each $\ell_i$ a positive real number. For $\ell\in \R^n_{>0}$ we define the \em planar polygon space \em $\mathcal{M}_\ell$ as
\begin{eqnarray*}
 \mathcal{M}_\ell&=&\left.\left\{ z\in T^n\,\left|\,\sum_{i=1}^n\ell_iz_i=0\right.\right\}\right/ SO(2),
\end{eqnarray*}
where $T^n=(S^1)^n$ is the standard torus and $SO(2)$ acts on $T^n$ by rotating each coordinate. The elements of $\mathcal{M}_\ell$ can be considered as closed configurations up to rotation and translation of a robot arm with $n$ bars, where $\ell$ determines the length of each bar. For this reason we also call $\ell$ a \em length vector\em.

Given a subset $J\subset \{1,\ldots,n\}$ we can look at the hyperplane $H_J$ defined by
\begin{eqnarray*}
 H_J&=&\left\{x\in \R^n\,\left|\,\sum\nolimits_{j\in J}x_j=\sum\nolimits_{j\notin J} x_j\right.\right\}.
\end{eqnarray*}
Then 
\begin{eqnarray*}
\R^n_{>0}-\bigcup_{J\subset \{1,\ldots,n\}}H_J
\end{eqnarray*}
consists of finitely many open components which we call \em chambers\em. It is possible to decompose $\R^n_{>0}$ further into strata, compare \cite{fahasc}, but here we will mainly be interested in chambers. It is known that if $\ell$ and $\ell'$ are in the same chamber, then $\mathcal{M}_\ell$ and $\mathcal{M}_{\ell'}$ are $O(1)$-diffeomorphic manifolds, where $O(1)\cong \Z/2$ acts on $\mathcal{M}_\ell$ by complex conjugation, see \cite{hausma}. Length vectors contained in a chamber are called \em generic\em.

Also, if $\sigma\in \Sigma_n$ is a permutation and $\ell'=(\ell_{\sigma(1)},\ldots,\ell_{\sigma(n)})$, we get a diffeomorphism $\mathcal{M}_\ell\cong \mathcal{M}_{\ell'}$ by permuting coordinates. We can therefore assume that $\ell$ is \em ordered\em, that is, we have $\ell=(\ell_1,\ldots,\ell_n)$ with $\ell_1\leq \ldots \leq \ell_n$.

The isomorphism problem for planar polygon spaces now asks whether diffeomorphic (or homeomorphic, or homotopy equivalent) polygon spaces are in the same chamber up to permutation.

The topology of planar polygon spaces has been considered by many authors, let us list Hausmann \cite{hausmf}, Kapovich and Millson \cite{kapmil}, Kamiyama, Tezuka and Toma \cite{kateto}, and Milgram and Trinkle \cite{miltri} to name a few.

In an unpublished undergraduate thesis \cite{walker}, Kevin Walker gave a formula for the homology groups of $\mathcal{M}_\ell$ for generic $\ell$. This formula was later recovered in \cite{farsch}; a reader-friendly approach can be found in \cite{faritr}. For $n\leq 5$ the Betti numbers of $\mathcal{M}_\ell$ determine the chamber up to permutation, which is easily seen by checking all possible chambers, but for $n\geq 6$ the Betti numbers alone do not determine the chamber. However, Walker asked whether the ring structure of $H^\ast(\mathcal{M}_\ell)$ determines the chamber up to permutation.

\begin{conjecture}[Walker \cite{walker}]\label{walkerconj}
 Let $\ell$, $\ell'$ be generic length vectors. If $H^\ast(\mathcal{M}_\ell)\cong H^\ast(\mathcal{M}_{\ell'})$ are isomorphic as graded rings, then $\ell$ and $\ell'$ are in the same chamber up to permutation.
\end{conjecture}

In order to tackle this conjecture, we need a sufficient understanding of the cohomology ring structure of these polygon spaces, together with a way to recover the chamber from the isomorphism class of the cohomology ring. Except in some special cases, the cohomology is not generated by elements of degree one. Instead we can look at the subring generated by elements of degree one, which is a quotient of an exterior algebra. For a certain class of such quotients it is possible to uniquely recover a minimal set of relations from the ring by a result of Gubeladze, see \cite{gubela} and \cite{brugub}.

We show that this strategy leads indeed to a positive answer of Walker's Conjecture.

\begin{theorem}
 Conjecture \ref{walkerconj} is true.
\end{theorem}

As a corollary we see that length vectors from different chambers do have different polygon spaces.

\begin{corollary}
 Let $\ell$ and $\ell'$ be generic length vectors with homotopy equivalent planar polygon spaces $\mathcal{M}_\ell$, $\mathcal{M}_{\ell'}$. Then $\ell$ and $\ell'$ are in the same chamber up to permutation.
\end{corollary}

In \cite{fahasc}, Farber, Hausmann and the author proved this for a large class of length vectors, for which the subring generated by elements of degree one is a quotient of an exterior algebra $\Lambda_\Z[X_1,\ldots,X_{n-1}]/I_\ell$. Here $I_\ell$ is the ideal generated by monomials $X_{i_1}\cdots X_{i_k}$ provided that $\{i_1,\ldots,i_k,n\}$ is long (see Definition \ref{shortlong} for the definition of long). Gubeladze's result applies exactly to ideals generated by monomials which then can be used to recover the chamber.

For the remaining cases the subring turns out to be more complicated. In particular, the quotient ideal is no longer generated by monomials only. Gubeladze's theorem can therefore not be applied directly. However, the ideal is still well behaved enough so that the problem can be reduced to a situation where this theorem applies.

The proof is organized as follows. In Section \ref{plapolspa} we analyze the missing cases and categorize them into several subcases which we call \em types\em. In Section \ref{bascohinv} we use rational cohomology and Betti numbers to show that different types lead indeed to different cohomology rings. The problem is therefore reduced to showing the statement of Walker's Conjecture for $\ell$ and $\ell'$ of the same type. This does not seem too promising at first, as we have $n-1$ different types for every $n$. However, for most of these types we can determine the cohomology by an induction argument on the type. This is done in Sections \ref{seci=1} and \ref{sectypen-4}. Section \ref{secn-4n-3} deals with another type and the remaining cases are fairly easy to handle.

To determine the cohomology we use the following strategy: the subring generated by elements of degree 1 or 0 is the quotient of an exterior algebra. We give a set of relations which we show to hold, and then show that there can be no more relations using a Morse theory argument. The necessary Morse theory is set up in Section \ref{diffeosec}. The relations come from two sources, one set of relations comes from a natural inclusion of the polygon space into a torus. The other set of relations comes from the fundamental group.

The relevant fundamental groups turn out to be not too complicated. They are either right-angled Artin groups or fairly simple iterated HNN-extensions of surface groups, so that their integral cohomology can be easily expressed in terms of relations. Again Morse theory is used to determine the fundamental groups.

Instead of planar polygon spaces one can look at higher dimensional analogues. For polygons in $\R^3$ the analogue of Walker's Conjecture was proven in \cite{fahasc} under the extra assumption $n>4$. The case $n=4$ does in fact not hold as there exist two different chambers for which the spatial polygon space is $S^2$. The proof is based on a description of the cohomology given in \cite{hauknu}, which can be modified so that Gubeladze's theorem is applicable. In higher dimensions the analogue of Walker's Conjecture is not known; it does hold however for a modified configuration space, see \cite{fhschn}.

\section{Planar polygon spaces}\label{plapolspa}

By rotating the last coordinate to 1, we can embed $\mathcal{M}_\ell$ into $T^{n-1}$ and we get
\begin{eqnarray*}
 \mathcal{M}_\ell&\cong&\left\{z\in T^{n-1}\,\left|\,\sum_{i=1}^{n-1}\ell_iz_i=\ell_n\right.\right\}.
\end{eqnarray*}

Complex conjugation in each coordinate induces involutions $\tau:\mathcal{M}_\ell\to \mathcal{M}_\ell$ and $\tau:T^{n-1}\to T^{n-1}$ which clearly commute with inclusion.

\begin{definition}\label{shortlong}
 Let $\ell\in \R^n_{>0}$ be a length vector. A subset $J\subset \{1,\ldots,n\}$ is called \em short with respect to $\ell$\em, if
\begin{eqnarray*}
 \sum\nolimits_{j\in J}\ell_j&<&\sum\nolimits_{j\notin J} \ell_j.
\end{eqnarray*}
Similarly $J$ is called \em long with respect to $\ell$\em, if the complement of $J$ is short with respect to $\ell$.
\end{definition}
By abuse of notation we will only write short, respectively long, if the length vector is clear. For a generic length vector $\ell$ all subsets $J\subset\{1,\ldots,n\}$ are either short or long, and the short subsets determine the chamber.

If $\ell$ is ordered, we define for $k\in\{0,\ldots,n-3\}$
\begin{eqnarray*}
 \dot{\mathcal{S}}_k(\ell)&=&\left\{\left.J\subset\{1,\ldots,n-1\}\,\right|\,J\cup\{n\}\mbox{ is short}\mbox{ and }|J|=k\right\}.
\end{eqnarray*}
If $\ell$ is not ordered, we define
\begin{eqnarray*}
 \dot{\mathcal{S}}_k(\ell)&=& \{\sigma^{-1}(J)\,|\, J\in \dot{\mathcal{S}}_k(\ell\sigma)\},
\end{eqnarray*}
where $\sigma\in \Sigma_n$ is a permutation such that $\ell\sigma=(\ell_{\sigma(1)},\ldots,\ell_{\sigma(n)})$ is ordered.

It is easy to see that for generic $\ell$ the chamber is determined by the sets $\dot{\mathcal{S}}_k(\ell)$. We get the following formula for the homology groups of $\mathcal{M}_\ell$:

\begin{theorem}[Walker \cite{walker}, Farber-Sch\"utz \cite{farsch}]\label{bettinumbers}
 Let $\ell$ be a generic length vector, and $a_k=|\dot{\mathcal{S}}_k(\ell)|$ for $k=0,\ldots,n-3$. Then
\begin{eqnarray*}
 H_k(\mathcal{M}_\ell)&=& \Z^{a_k}\oplus \Z^{a_{n-3-k}}
\end{eqnarray*}
for all $k=0,\ldots,n-3$.
\end{theorem}

If $\ell$ is not generic, $H_k(\mathcal{M}_\ell)$ has an extra $\Z$-summand for every subset $J\subset \{1,\ldots,n-1\}$ with $k$ elements such that $J\cup \{n\}$ is neither short nor long, compare \cite{farsch} (assuming $\ell_n$ to be maximal).

\begin{definition}
 Let $\ell$ be a generic length vector. Then $\ell$ is called \em normal\em, if $\dot{\mathcal{S}}_{n-4}(\ell)=\emptyset$.
\end{definition}

\begin{theorem}[Farber-Hausmann-Sch\"utz \cite{fahasc}]\label{normalwalker}
 Let $\ell$, $\ell'$ be generic length vectors with $\ell$ normal. If $H^\ast(\mathcal{M}_\ell)\cong H^\ast(\mathcal{M}_{\ell'})$ are isomorphic as graded rings, then $\ell$ and $\ell'$ are in the same chamber up to permutation.
\end{theorem}

Intuitively, it is clear that for large values of $n$ only a small proportion of length vectors are not normal. For a more precise statement see \cite[Prop.6.2]{farbrl}.

It is known that there is only one chamber up to permutation with $\dot{\mathcal{S}}_{n-3}(\ell)\not=\emptyset$, see, for example, \cite[Ex.2]{farsch}, so Conjecture \ref{walkerconj} is reduced to dealing with length vectors satisfying $\dot{\mathcal{S}}_{n-3}(\ell)=\emptyset\not=\dot{\mathcal{S}}_{n-4}(\ell)$.

\begin{definition}
 A generic length vector $\ell$ is called \em special\em, if $\dot{\mathcal{S}}_{n-3}(\ell)=\emptyset\not=\dot{\mathcal{S}}_{n-4}(\ell)$.
\end{definition}

It follows from \cite[Lm.7]{fahasc} that special length vectors $\ell$ are indeed detected by the cohomology of $\mathcal{M}_\ell$. Note that the proof of \cite[Lm.7]{fahasc} simplifies if one only considers generic length vectors.

The power set of $\{1,\ldots,n-1\}$, $\mathcal{P}(\{1,\ldots,n-1\})$, can be given the following partial order, compare \cite{haurod}. We say $J_1\leq J_2$, if there exists an order-preserving injective function $\varphi:J_1\to J_2$ with $x\leq\varphi(x)$ for all $x\in J_1$.

Let
\begin{eqnarray*}
 \dot{\mathcal{S}}(\ell)&=&\bigcup_{k=0}^{n-3} \dot{\mathcal{S}}_k(\ell).
\end{eqnarray*}
If $\ell$ is ordered and $J\in\dot{\mathcal{S}}(\ell)$, then $I\in \dot{\mathcal{S}}(\ell)$ for all $I\leq J$. In particular the chamber of a generic and ordered length vector $\ell$ is determined by the maximal elements of $(\dot{\mathcal{S}}(\ell),\leq)$.

\begin{lemma}\label{totalorder}
 Let $\ell$ be an ordered, special length vector. Then $(\dot{\mathcal{S}}_{n-4}(\ell),\leq)$ is totally ordered.
\end{lemma}

\begin{proof}
Let $J_1,J_2\in \dot{\mathcal{S}}_{n-4}(\ell)$ and let $\bar{J}_1$, $\bar{J}_2$ denote the complements in $\{1,\ldots,n-1\}$. Then $\bar{J}_1$ and $\bar{J}_2$ are long with respect to $\ell$. Write $J_1=\{i_1,i_2,i_3\}$ with $i_1<i_2<i_3$ and $J_2=\{j_1,j_2,j_3\}$ with $j_1<j_2<j_3$. It is easy to see that $J_1\leq J_2$ is equivalent to $i_r\leq j_r$ for $r=1,2,3$.

If $\bar{J}_1\cap \bar{J}_2$ contains at least two elements, we get $J_1\leq J_2$ or $J_2\leq J_1$, so assume $|\bar{J}_1\cap \bar{J}_2|\leq 1$. We get that one of the sets $\{i_1,i_2,n\}, \{i_1,i_3,n\}$ or $\{i_2,i_3,n\}$ is a subset of $J_2\cup\{n\}$. Therefore $\{i_1,i_2,n\}$ is short, but this implies that $\{i_1,i_2,i_3\}\leq \{i_1,i_2,n\}$ is also short, which is a contradiction to $J_1\in\dot{\mathcal{S}}_{n-4}(\ell)$.
\end{proof}

\begin{definition}
 Let $\ell$ be an ordered special length vector. The \em type of $\ell$ \em is defined as the complement of the maximal element of $\dot{\mathcal{S}}_{n-4}(\ell)$. In other words, the type is the minimal subset of $\{1,\ldots,n-1\}$ with three elements that is long.
\end{definition}

\begin{remark}\label{alltypes}
 If $\ell$ is ordered and special, we always get that $\{n-3,n-2,n-1\}=\bar{J}$ is long. Furthermore, it follows from the proof of Lemma \ref{totalorder} that any other long subset of $\{1,\ldots,n-1\}$ with three elements has two elements with $\bar{J}$ in common. Also note that it is not possible that $\{n-5,n-3,n-1\}$ is long, as this would imply that $\{n-4,n-2,n\}$ is short as a subset of the complement, even though $\{n-5,n-3,n-1\}\leq \{n-4,n-2,n\}$.

Therefore the only possible types are
\[
 \{n-3,n-2,n-1\},\{n-4,n-2,n-1\},\ldots,\{1,n-2,n-1\}
\]
and
\[
 \{n-4,n-3,n-1\}, \{n-4,n-3,n-2\}.
\]
\end{remark}

\section{Basic cohomology invariants of planar linkages}\label{bascohinv}

\begin{definition}
 Let $\Delta$ be a finite simplicial complex with ordered set of vertices $v_1,\ldots,v_k$ and $R$ a commutative ring. The \em exterior face ring \em $\Lambda_R[\Delta]$ is the quotient of the exterior algebra $\Lambda_R[X_1,\ldots,X_k]$ by the ideal generated by the monomials $X_{i_1}\wedge\ldots\wedge X_{i_l}$ with $\{v_{i_1},\ldots,v_{i_l}\}\notin \Delta$.
\end{definition}

We will usually omit the '$\wedge$' symbol when writing products.

The following algebraic fact was proven by Gubeladze in \cite{gubela} for $R$ of characteristic 2, the case of arbitrary commutative rings can be found in \cite[Exercise 5.12]{brugub}. Note that \cite{gubela} deals with symmetric algebras and ideals generated by monomials in the $X_1,\ldots, X_k$. If $R$ has characteristic 2, the exterior algebra can be obtained from the symmetric algebra by adding the relations $X_i^2$.

\begin{theorem}[Gubeladze \cite{gubela}, Bruns-Gubeladze \cite{brugub}]\label{isoprobsol}
 Let $R$ be a commutative ring and $\Delta$, $\Delta'$ finite simplicial complexes with $\Lambda_R[\Delta]\cong \Lambda_R[\Delta']$. Then $\Delta\cong \Delta'$ as simplicial complexes.\hfill \qedsymbol
\end{theorem}

Note that $\tilde{\mathcal{S}}(\ell)=\dot{\mathcal{S}}(\ell)-\{\emptyset\}$ with inclusion is a finite simplicial complex with vertex set $\dot{\mathcal{S}}_1(\ell)$ so we can form an exterior face ring $\Lambda_R[\tilde{\mathcal{S}}(\ell)]$.

Recall that we have an inclusion $i:\mathcal{M}_\ell\to T^{n-1}$ obtained from rotating $z_n$ to 1. Define
\begin{eqnarray*}
 B^\ast_\ell&=&{\rm im}(i^\ast:H^\ast(T^{n-1};\Z)\to H^\ast(\mathcal{M}_\ell;\Z)).
\end{eqnarray*}

The following proposition follows from \cite[Thm.6]{fahasc}.

\begin{proposition}\label{thefacering}
 Let $\ell$ be a generic length vector. Then $B_\ell^\ast \cong \Lambda_\Z[\tilde{\mathcal{S}}(\ell)]$. Furthermore, $B_\ell^\ast$ consists exactly of those  $x\in H^\ast(\mathcal{M}_\ell;\Z)$, for which $\tau^\ast(x)=(-1)^{|x|}x$, where $|x|$ denotes the degree of the cohomology class, and $\tau:\mathcal{M}_\ell\to \mathcal{M}_\ell$ is the involution given by complex conjugation in every variable.
\end{proposition}

We call $B_\ell^\ast$ the \em balanced subalgebra \em of $H^\ast(\mathcal{M}_\ell)$. The proof of Theorem \ref{normalwalker} is based on the fact that a cohomology isomorphism induces an isomorphism of balanced subalgebras for normal length vectors, but this argument does not carry over to the special length vectors.

\begin{lemma}\label{tauaction}
 The action of the involution $\tau_\ast$ on $H_i(T^{n-1})$ and $H_i(T^{n-1},\mathcal{M}_\ell)$, and the action of $\tau^\ast$ on $H^i(T^{n-1})$ and $H^i(T^{n-1},\mathcal{M}_\ell)$ coincide with $(-1)^i$, for all $i=0,\ldots,n-1$.
\end{lemma}

\begin{proof}
 The cohomological version of this Lemma was proven in \cite[Lemma 2]{fahasc}. Since all homology groups are finitely generated free abelian, the result follows from the universal coefficient theorem.
\end{proof}

If we look at cohomology with rational coefficients, we thus get
\begin{eqnarray*}
H^\ast(\mathcal{M}_\ell;\Q)&=&B_\ell^\ast\oplus C_\ell^\ast
\end{eqnarray*}
where $\tau^\ast(e)=(-1)^{|e|+1}e$ for all $e\in C_\ell^\ast$.

\begin{lemma}\label{tauonN}
 Let $\ell\in\R^n_{>0}$ be generic and such that $\mathcal{M}_\ell$ is connected. Then $\tau_\ast(x)=(-1)^{n-2}x$ for all $x\in H_{n-3}(\mathcal{M}_\ell)$.
\end{lemma}

\begin{proof}
 We have the long exact sequence
\[
 \ldots\longrightarrow H_{n-2}(T^{n-1},\mathcal{M}_\ell)\longrightarrow H_{n-3}(\mathcal{M}_\ell) \stackrel{i_\ast}{\longrightarrow} H_{n-3}(T^{n-1})\longrightarrow \ldots
\]
By Lemma \ref{tauaction}, $\tau_\ast$ acts on $H_i(T^{n-1})$ and $H_i(T^{n-1},\mathcal{M}_\ell)$ as multiplication by $(-1)^i$. As $\mathcal{M}_\ell$ is connected, all $x\in H_{n-3}(\mathcal{M}_\ell)$ are multiples of the fundamental class. Now the dual space $(H_{n-3}(T^{n-1}))^\ast\cong H_2(T^{n-1})$ and the latter is generated by 2-tori which are realized by fixing all but two coordinates $i,j$ as $e_1$. For such a 2-torus $T^2$ we get $[\mathcal{M}_\ell]\cdot [T^2]=0$, as $\mathcal{M}_\ell\cap T^2=\emptyset$, since $\{i,j\}$ is a short subset by the assumption that $\mathcal{M}_\ell$ is connected. It follows that $i_\ast[\mathcal{M}_\ell]=0$, and the statement follows by naturality and Lemma \ref{tauaction}.
\end{proof}

Recall that $B^\ast_\ell$ is generated by elements $X_i\in H^1(\mathcal{M}_\ell)$. For $J\in \dot{\mathcal{S}}_k(\ell)$ we write $X_J=X_{i_1}\cdots X_{i_k}$, where $i_1<\ldots<i_k$ satisfies $J=\{i_1,\ldots,i_k\}$.

\begin{proposition}\label{basis}
$C_\ell^k$ is generated by elements $Y_J\in C_\ell^k$ for $J\in \dot{\mathcal{S}}_{n-3-k}(\ell)$ which satisfy
\begin{eqnarray*}
X_I\cup Y_J&=&\left\{\begin{array}{cl}
\pm Y_{J-I}&\mbox{if }I\subset J\\
0&\mbox{else}\end{array}\right.
\end{eqnarray*}
\end{proposition}

\begin{proof}
Define $Y_\emptyset\in H^{n-3}(\mathcal{M}_\ell;\Q)$  to be the Poincar\'e dual of a point. By Lemma \ref{tauonN} we have $\tau_\ast([\mathcal{M}_\ell])=(-1)^{n-2}[\mathcal{M}_\ell]$ for the fundamental class, so we get $\tau^\ast(Y_\emptyset)=(-1)^{n-2}Y_\emptyset$, that is, $Y_\emptyset\in C^{n-3}_\ell$.

By Poincar\'e duality we have $H^i(\mathcal{M}_\ell;\Q)\cong {\rm Hom}_\Q(H^{n-3-i}(\mathcal{M}_\ell;\Q),\Q)$, so for $I\in\dot{\mathcal{S}}_i(\ell)$ define $Y_I\in H^{n-3-i}(\mathcal{M}_\ell;\Q)$ by
\begin{eqnarray*}
X_I\cup Y_I&=&Y_\emptyset\\
X_J\cup Y_I&=&0\mbox{ for }J\not=I\\
e\cup Y_I&=&0\mbox{ for }e\in C_\ell^{n-3-i}
\end{eqnarray*}
Note that for $b\in B^{n-3-i}_\ell$ we have $b\cup X_K=0$ for all $K\in\dot{\mathcal{S}}_i(\ell)$, so the $Y_I$ together with $X_K$ give a basis for the cohomology. Also
\begin{eqnarray*}
X_I\cup \tau^{\ast}(Y_I)&=&\tau^\ast(\tau^\ast(X_I)\cup Y_I)\\
&=&(-1)^i\tau^\ast Y_\emptyset\\
&=&(-1)^{n-2-i} Y_\emptyset.
\end{eqnarray*}
Other cup-products with $\tau^\ast(Y_I)$ are 0, so $\tau^\ast(Y_I)=(-1)^{n-2-i}Y_I$ and $Y_I\in C_\ell^{n-3-i}$.

As $\tau^\ast(X_I\cup Y_J)=\tau^\ast(X_I)\cup \tau^\ast(Y_J)$, we get that $X_I\cup Y_J\in C^\ast_\ell$ for arbitrary $I,J\in \dot{\mathcal{S}}(\ell)$.

If $I\subset J$, then
\begin{eqnarray*}
X_{J-I}\cup (X_I\cup Y_J)&=&\pm X_J\cup Y_J\\
&=& \pm Y_\emptyset
\end{eqnarray*}
and for all other $K\subset \dot{\mathcal{S}}_{|J-I|}(\ell)$ we have $X_K\cup X_I\cup Y_J=0$. Also, if $e\in C_\ell^{|J-I|}$, then $e\cup X_I\in C^{|J|}_\ell$ and $e\cup X_I\cup Y_J=0$. Therefore $X_I\cup Y_J=\pm Y_{J-I}$.

If $I\not\subset J$ and $K\in \dot{\mathcal{S}}_{|J|-|I|}(\ell)$, then $K\cup I\not=J$, so $X_K\cup X_I\cup Y_J=0$. As before $e\cup X_I\cup Y_J=0$ for $e\in C^{|J|-|I|}$. Thus $X_I\cup Y_J=0$.
\end{proof}

Define
\begin{eqnarray*}
A^k_i(\ell)&=&\{x\in H^k(\mathcal{M}_\ell;\Q)\,|\,x\cdot y_1\cdots y_i=0 \mbox{ for all }y_1,\ldots,y_i\in H^1(\mathcal{M}_\ell;\Q)\}.
\end{eqnarray*}

\begin{lemma}\label{annihi}
If $k+i\leq n-3$, then
\begin{eqnarray*}
A^k_i(\ell)&=&\langle X_I\,|\, I\in \dot{\mathcal{S}}_k(\ell), I\not\subset J \mbox{ for all }J\in\dot{\mathcal{S}}_{k+i}(\ell)\rangle.
\end{eqnarray*}
\end{lemma}

\begin{proof}
It is easy to see that $X_I\in A^k_i$ if $I\in \dot{\mathcal{S}}_k(\ell)$ with $I\not\subset J$ for all $J\in \dot{\mathcal{S}}_{k+i}(\ell)$.

Now if $x\in A^k_i$, then
\begin{eqnarray*}
x&=&\sum_{I\in\dot{\mathcal{S}}_k(\ell)}n_IX_I + \sum_{J\in\dot{\mathcal{S}}_{n-3-k}(\ell)}m_JY_J.
\end{eqnarray*}
Let $K\in\dot{\mathcal{S}}_i(\ell)$. As $X_K$ is the product of $i$ elements, we get
\begin{eqnarray*}
x\cdot X_K & = & b'+\sum_{J\in\dot{\mathcal{S}}_{n-3-k}(\ell), K\subset J}\pm m_J Y_{J-K}\\
&=&0
\end{eqnarray*}
with $b'\in B^{i+k}_\ell$, so in particular $b'=0$.

The $Y_{J-K}$ are all linearly independent and therefore $m_J=0$ for all $J$ appearing in the sum. As $i\leq n-3-k$, we can find $K\subset J$ with $i$ elements for every $J\in \dot{\mathcal{S}}_{n-3-k}(\ell)$. This shows that $x\in B^k_\ell$.

Now if $I\subset J$ for some $J\in\dot{\mathcal{S}}_{k+i}(\ell)$ we get $n_I=0$ as we can multiply $x$ with $X_{J-I}$ which is a product of $i$ elements.
\end{proof}

In other words, $A^k_i(\ell)$ is generated by those sets $I$ with $k$ elements, such that $I\cup \{n\}$ is short, but cannot be extended to a short set with $i+k+1$ elements (including $n$).

\begin{lemma}\label{bettispecial}
 Let $\ell$ be an ordered, special length vector, and let $d_1(\ell)=\dim A^1_{n-5}(\ell)$, $d_2(\ell)=\dim A^2_{n-6}(\ell)$ and $d_3(\ell)=\dim A^{n-5}_1(\ell)$. Then
\begin{enumerate}
 \item If $\ell$ is of type $\{n-4,n-3,n-2\}$ then $d_1(\ell)=d_2(\ell)=d_3(\ell)=0$ and we have
\begin{eqnarray*}
b_1(\mathcal{M}_\ell)&=&n+3 \\ b_2(\mathcal{M}_\ell)&=&\begin{pmatrix} n-5 \\ 2 \end{pmatrix}+8(n-5)+1.
\end{eqnarray*}
\item If $\ell$ is of type $\{n-4,n-3,n-1\}$ we have
\begin{eqnarray*}
 b_1(\mathcal{M}_\ell)&=&n+1+d_1(\ell) \\ b_2(\mathcal{M}_\ell)&=&\begin{pmatrix} n-5 \\ 2 \end{pmatrix}+6(n-5)+1+d_2(\ell)+d_3(\ell).
\end{eqnarray*}
\item If $\ell$ is of type $\{n-3,n-2,n-1\}$, we have
\begin{eqnarray*}
 b_1(\mathcal{M}_\ell)&=&n-3+d_1(\ell).
\end{eqnarray*}
\item If $\ell$ is of type $\{i,n-2,n-1\}$ with $i\in\{1,\ldots,n-4\}$ we have
\begin{eqnarray*}
b_1(\mathcal{M}_\ell)&=&2n-5-i+d_1(\ell).
\end{eqnarray*}
Furthermore, if $i\leq n-5$ we have
\begin{eqnarray*}
b_2(\mathcal{M}_\ell)&=&\begin{pmatrix}n-3 \\ 2 \end{pmatrix}+(n-4)+(n-5)+\ldots +(i-1)+d_2(\ell)+d_3(\ell).
\end{eqnarray*}
\end{enumerate}

\end{lemma}

\begin{proof}
 Assume $\ell$ is of type $\{n-4,n-3,n-2\}$. Then $\{1,2,\ldots,n-5,n-1,n\}$ is short. Thus $a_1(\ell)=n-1$ and $d_1(\ell)=0$. Clearly $a_{n-4}=4$ which shows that $b_1=n+3$ by Theorem \ref{bettinumbers}. If $\{i,j,n\}$ is short, only one of $i$ and $j$ can be $\geq n-4$, since $\{n-4,n-3,n-2\}$ is long. Hence $d_2(\ell)=0$ and $a_2(\ell)=\begin{pmatrix} n-5 \\ 2 \end{pmatrix} +4(n-5)$, the first summand corresponding to pairs with $i,j\leq n-5$, and the second corresponding to pairs with $i\in\{n-4,n-3,n-2,n-1\}$. If $J\in \dot{\mathcal{S}}_{n-5}(\ell)$, then $J\subset \{1,2,\ldots,n-5,i\}$ with $i\in \{n-4,n-3,n-2,n-1\}$. Therefore $a_{n-5}(\ell)=4(n-5)+1$ and $d_3(\ell)=0$. The result on $b_2$ follows.

Now let $\ell$ be of type $\{n-4,n-3,n-1\}$. Then $\{1,2,\ldots,n-5,n-2,n\}$ is short and therefore $\{n-2,n\}$ is short. It is possible that $\{n,n-1\}$ is short, but in that case $X_{n-1}$ generates $A^1_{n-4}(\ell)$ by Lemma \ref{annihi}. Thus $a_1(\ell)=n-2+d_1(\ell)$. As $a_{n-4}=3$, we get $b_1=n+1+d_1$.For $\dot{\mathcal{S}}_2(\ell)$, we get $d_2$ pairs which cannot be extended to elements of $\dot{\mathcal{S}}_{n-4}(\ell)$, the others can be calculated as above to give $a_2(\ell)=\begin{pmatrix} n-5 \\ 2 \end{pmatrix} +3(n-5)+d_2$. Note that the factor 3 corresponds to the fact that one element being $n-1$ can only contribute to $d_2$. If $J\in \dot{\mathcal{S}}_{n-5}(\ell)$, it either contributes to $d_3$, or $J\subset \{1,2,\ldots,n-5,i\}$ with $i\in \{n-4,n-3,n-2\}$. As above we get $a_{n-5}(\ell)=3(n-5)+1+d_3$, which gives the result for $b_2$.

If $\ell$ is of type $\{n-3,n-2,n-1\}$, we have $\{n,n-4\}$ short. If $\{n,i\}$ is short for $i\in \{n-3,n-2,n-1\}$, it will contribute to $d_1(\ell)$. Since $a_{n-4}=1$, we get $b_1=n-4+d_1(\ell)+1$.

If $\ell$ is of type $\{i,n-2,n-1\}$ with $i\leq n-4$ we get $a_1=n-3+d_1$ and $a_{n-4}=n-2-i$, so $b_1=2n-5-i+d_1(\ell)$. If $i\leq n-5$, we get $\begin{pmatrix} n-3 \\ 2 \end{pmatrix}$-pairs which are a subset of some $J\in \dot{\mathcal{S}}_{n-4}(\ell)$, all other pairs in $\dot{\mathcal{S}}_2(\ell)$ contribute to $d_2(\ell)$. Hence $a_2=\begin{pmatrix} n-3 \\ 2 \end{pmatrix}+d_2(\ell)$. If $J\subset \dot{\mathcal{S}}_{n-5}(\ell)$, it either contributes to $d_3(\ell)$, or we obtain $J$ from $\{1,\ldots,n-3\}$ by removing two elements, one of which is greater equal to $i$. There are $(n-4)+(n-5)+\ldots +(i-1)$ such sets. Hence $a_{n-5}(\ell)=(n-4)+(n-5)+\ldots +(i-1)+d_3(\ell)$ and the result follows.
\end{proof}

\begin{corollary}\label{sametype}
 Let $\ell$, $\ell'$ be ordered, special length vectors with $H^\ast(\mathcal{M}_\ell;\Q)\cong H^\ast(\mathcal{M}_{\ell'};\Q)$. Then $\ell$ and $\ell'$ are of the same type.
\end{corollary}

\begin{proof}
The cohomology isomorphism implies that $b_i(\mathcal{M}_\ell)=b_i(\mathcal{M}_{\ell'})$ for all $i$ and $d_i(\ell)=d_i(\ell')$ for $i=1,2,3$ as in Lemma \ref{bettispecial}.

If $\ell$ is of type $\{i,n-2,n-1\}$ with $i\leq n-3$ and $\ell'$ is of type $\{i',n-2,n-1\}$ with $i'\leq n-3$, it follows from Lemma \ref{bettispecial} (3) and (4) that $i=i'$.

If $\ell$ is of type $\{n-4,n-3,n-2\}$, then $d_1(\ell)=0$, and it follows from the first Betti number calculated in Lemma \ref{bettispecial} that $\ell'$ can only be of type $\{n-4,n-3,n-2\}$ or $\{n-8,n-2,n-1\}$. Since $d_2(\ell)=d_3(\ell)=0$ the same holds for $d_i(\ell')$, $i=2,3$. If $\ell'$ were of type $\{n-8,n-2,n-1\}$, we get
\begin{eqnarray*}
 b_2(\mathcal{M}_{\ell'})&=& \begin{pmatrix}n-3 \\ 2 \end{pmatrix}+(n-4)+(n-5)+\ldots +(n-9)
\end{eqnarray*}
by Lemma \ref{bettispecial} (4). An elementary calculation shows that $b_2(\mathcal{M}_\ell)-b_2(\mathcal{M}_{\ell'})=9$, a contradiction. Therefore $\ell'$ is also of type $\{n-4,n-3,n-2\}$.

Assume $\ell$ is of type $\{n-4,n-3,n-1\}$. By checking the first Betti number, we see that $\ell'$ either has the same type as $\ell$, or is of type $\{n-6,n-2,n-1\}$. If $\ell$ were of type $\{n-6,n-2,n-1\}$, an elementary calculation, using Lemma \ref{bettispecial} (2) and (4) shows that $b_2(\mathcal{M}_\ell)-b_2(\mathcal{M}_{\ell'})=2$, a contradiction. So $\ell'$ has the same type as $\ell$.
\end{proof}

\begin{definition}
 A type $T$ of special length vectors is called \em cohomologically rigid\em, if whenever $\ell$, $\ell'$ are of type $T$ with $H^\ast(\mathcal{M}_\ell;\Z)\cong H^\ast(\mathcal{M}_{\ell'};\Z)$, we have $\ell$ and $\ell'$ are in the same chamber up to permutation.
\end{definition}

By Corollary \ref{sametype}, Walker's Conjecture is reduced to proving that all types of special length vectors are cohomologically rigid. We begin with a rather simple case.

\begin{lemma}\label{onelowtype}
 There exists exactly one chamber of type $\{n-4,n-3,n-2\}$ up to permutation.
\end{lemma}

\begin{proof}
 Let $J\in\dot{\mathcal{S}}(\ell)$ and let $m=\max J$. If $m\leq n-4$, then clearly $J\cup \{n\} \leq \{1,\ldots,n-5,n-1,n\}$, so $J$ is determined by the type. If $m > n-4$, let $\tilde{m}=\max J-\{m\}$. It is not possible that $\tilde{m}\geq n-4$, for then $\{n-4,n-3,n-2\}\leq \{\tilde{m},m,n\}$ would be short. But then $J\cup\{n\}\leq \{1,\ldots,n-5,n-1,n\}$, so $J$ is determined by the type.
\end{proof}

So in particular, type $\{n-4,n-3,n-2\}$ is cohomologically rigid.

\begin{theorem}\label{rigidn-3}
 Type $\{n-3,n-2,n-1\}$ is cohomologically rigid.
\end{theorem}

\begin{proof}
 For any $\ell$ let $H^\ast_{(1)}(\mathcal{M}_\ell;\Q)$ be the subalgebra of $H^\ast(\mathcal{M}_\ell;\Q)$ generated by $H^0(\mathcal{M}_\ell;\Q)$ and $H^1(\mathcal{M}_\ell;\Q)$. If $\ell$ and $\ell'$ are ordered special length vectors of type $\{n-3,n-2,n-1\}$ with $H^\ast(\mathcal{M}_\ell)\cong H^\ast(\mathcal{M}_{\ell'})$, then $H^\ast_{(1)}(\mathcal{M}_\ell;\Q)\cong H^\ast_{(1)}(\mathcal{M}_{\ell'};\Q)$.

Note that $H^1(\mathcal{M}_\ell;\Q)$ is generated by $X_i$ with $i\in \{1,\ldots,n-1\}$ such that $\{i,n\}$ is short with respect to $\ell$, and one extra generator $Y$. It follows from Proposition \ref{basis} that $H^\ast_{(1)}(\mathcal{M}_\ell;\Q)$ is an exterior face ring with respect to the following simplicial complex: let $\Delta_\ell$ be $\tilde{\mathcal{S}}(\ell)$ together with one extra vertex corresponding to $Y$ which cones off $X_1,\ldots, X_{n-4}$.

Denote the generators of $H^\ast_{(1)}(\mathcal{M}_{\ell'};\Q)$ by $X_i'$ and $Y'$.

We then have $\Lambda_\Q[\Delta_\ell]\cong \Lambda_\Q[\Delta_{\ell'}]$. By Theorem \ref{isoprobsol}, we get an isomorphism $\Delta_\ell\cong \Delta_{\ell'}$ of simplicial complexes, that is, we have a bijection $\sigma:\{X_1,\ldots,X_k,Y\} \to \{X_1',\ldots,X_k',Y'\}$ preserving simplices. Here $k$ is the largest number $\leq n-1$  such that $\{k,n\}$ is short with respect to $\ell$ (and hence also with respect to $\ell'$). Note that $Y$ cones off the simplex $X_1,\ldots, X_{n-4}$, and if $X_i'$ with $i\leq n-4$ has the property that it is adjacent to $X_j'$ if and only if $j\leq n-4$, it is conceivable that $\sigma(Y)=X_i'$. But then there has to be an $X_l$ with $\sigma(X_l)=Y'$, and $X_l$ is adjacent exactly to all $X_j$ with $j\leq n-4$. Define $\sigma':\{X_1,\ldots,X_k,Y\}\to \{X_1',\ldots, X_k',Y\}$ by $\sigma'(X_s)=\sigma(X_s)$ for $s\not= l$, $\sigma'(X_l)=X_i'$ and $\sigma'(Y)=Y'$. Then $\sigma'$ induces an isomorphism $\tilde{\mathcal{S}}(\ell)\cong \tilde{\mathcal{S}}(\ell')$, which by \cite[Lm.3]{fahasc} implies $\dot{\mathcal{S}}(\ell)=\dot{\mathcal{S}}(\ell')$. Therefore $\ell$ and $\ell'$ are in the same chamber.
\end{proof}

\begin{proposition}\label{lamerigid}
 Type $\{1,n-2,n-1\}$ is cohomologically rigid.
\end{proposition}

\begin{proof}
 It is easy to see that there is only one chamber with $A^1_{n-5}=0$. More generally, we can have $A^1_{n-5}$ zero-, one- or two-dimensional, depending on whether $\{n,n-1\}$ and $\{n,n-2\}$ are short or long. But since $\{1,n-2,n-1\}$ is long, we get $\{n,n-2,1\}$ long, so there exist only three chambers up to permutation of special length vectors of type $\{1,n-2,n-1\}$. Furthermore, these three chambers are distinguished by the first Betti number, see Lemma \ref{bettispecial}.
\end{proof}

A similar argument can be used to show that type $\{2,n-2,n-1\}$ is cohomologically rigid, but the arguments using the annihilators $A^i_k$ do not work for type $\{i,n-2,n-1\}$ with larger values of $i$. We therefore need different techniques, in particular we need to get a better picture of the algebra generated by $H^0(\mathcal{M}_\ell;\Z)$ and $H^1(\mathcal{M}_\ell;\Z)$ which we denote by
\[
 H^\ast_{(1)}(\mathcal{M}_\ell;\Z).
\]
Since the homology of $\mathcal{M}_\ell$ is free abelian, we necessarily get that
\begin{eqnarray*}
 H^\ast_{(1)}(\mathcal{M}_\ell;\Z)& \cong & \Lambda_\Z[A_1,\ldots,A_k]/I_\ell
\end{eqnarray*}
where $k$ is the first Betti number and $I_\ell$ is some ideal of $\Lambda_\Z[A_1,\ldots,A_k]$ whose elements have at least degree 2. We will see below that $H^\ast_{(1)}(\mathcal{M}_\ell;\Z)$ need not be a face ring in general, but $I_\ell$ is still well behaved enough to recover $\dot{\mathcal{S}}(\ell)$.

\section{The diffeomorphism type of certain planar polygon spaces}\label{diffeosec}

Let $\ell=(\ell_1,\ldots,\ell_n)$ be a generic ordered length vector. Given $\varepsilon>0$, we can form a new length vector $\ell'=(\varepsilon,\ell_1,\ldots,\ell_n)$. If $\varepsilon>0$ is sufficiently small, a set $J\subset \{1,\ldots,n+1\}$ is small with respect to $\ell'$ if and only if $s(J-\{1\})$ is small with respect to $\ell$, where $s:\{2,\ldots,n+1\}\to \{1,\ldots,n\}$ is given by $s(i)=i-1$.

\begin{proposition}\label{smallfirst}
 In the above situation, $\mathcal{M}_{\ell'}\cong \mathcal{M}_\ell\times S^1$, and the $S^1$-factor corresponds to rotating the first coordinate.
\end{proposition}

\begin{proof}
 Let $f:T^{n-1}\to \C$ be given by
\begin{eqnarray*}
f(z_1,\ldots,z_{n-1})&=&\sum_{i=1}^{n-1}\ell_i z_i+\ell_n.
\end{eqnarray*}
Then $\mathcal{M}_\ell=f^{-1}(\{0\})$, and it is easy to see that 0 is a regular value of $f$. So the normal bundle of $\mathcal{M}_\ell$ in $T^{n-1}$ is trivial, and we get $\mathcal{M}_{\ell'}$ is diffeomorphic to $f^{-1}(S_\varepsilon)$, where $S_\varepsilon\subset \C$ is the circle around 0 of radius $\varepsilon$. So for $\varepsilon>0$ small enough, we get $\mathcal{M}_{\ell'}\cong \mathcal{M}_\ell\times S^1$.
\end{proof}

\begin{corollary}\label{thelowtype}
 If $\ell$ is special of type $\{n-4,n-3,n-2\}$ with $n\geq 5$, then $\mathcal{M}_\ell\cong M_4\times T^{n-5}$, where $M_4$ is the closed orientable surface of genus 4.
\end{corollary}

\begin{proof}
 We have $\ell=(\varepsilon,\ldots,\varepsilon,1,1,1,1,1)$ is of type $\{n-4,n-3,n-2\}$, provided $\varepsilon>0$ is small enough. As there is only one chamber of type $\{n-4,n-3,n-2\}$ up to permutation by Lemma \ref{onelowtype}, the result follows by induction, the induction start $n=5$ is easily seen using the Betti numbers.
\end{proof}

\begin{corollary}
 If $\mathcal{M}_\ell$ is disconnected, then $\mathcal{M}_\ell\cong S^0\times T^{n-3}$.
\end{corollary}

\begin{proof}
 It follows from Theorem \ref{bettinumbers} that we need $\{n-2,n-1\}$ to be long with respect to an ordered length vector for $\mathcal{M}_\ell$ to be disconnected. Then there can only be one chamber up to permutation with this property, and it is easy to see that $(\varepsilon,\ldots,\varepsilon,1,1,1)$ represents this chamber for sufficiently small $\varepsilon$. The result now follows by induction, starting with $n=3$.
\end{proof}

This corollary is well known, a proof can already be found in \cite{walker}. See also \cite{kapmil} and \cite{farsch}.

Now consider a generic ordered length vector $\ell=(\ell_1,\ldots,\ell_n)$ and for $t\in[0,\ell_1)$ let $\ell_t=(\ell_1-t,\ell_2,\ldots,\ell_{n-1},\ell_n+t)$. Then $\ell_t$ is also ordered, and a subset $J\subset\{2,\ldots,n-1\}$ is short with respect to $\ell_t$ for any $t$ if and only if it is short with respect to $\ell$.

Write $T^{n-2}=\{(z_2,\ldots,z_{n-1})\,|\, z_i\in S^1 \}$ and define $f:T^{n-2}\times [0,\ell_1-\varepsilon]\to \R$ by
\begin{eqnarray*}
 f(z_2,\ldots,z_{n-1},t)&=& \left\|\sum_{i=2}^{n-1} \ell_i z_i+\ell_n+t \right\|^2-(\ell_1-t)^2,
\end{eqnarray*}
where $\varepsilon>0$ is so small that $J\subset \{1,\ldots,n\}$ is small with respect to $\ell_{\ell_1-\varepsilon}$ if and only if $s(J-\{1\})$ is small with respect to $\ell'=(\ell_2,\ldots,\ell_{n-1},\ell_n+\ell_1)$.

\begin{lemma}\label{critpointsf}
 If $\ell$ is generic, then $0$ is a regular value of $f$ and of $f|\partial(T^{n-2}\times [0,\ell_1-\varepsilon])$.
\end{lemma}

\begin{proof}
 Using angular coordinates, it is easy to see that
\begin{eqnarray*}
 f(\theta_2,\ldots,\theta_{n-1},t)&=&\|\ell\|^2-2\ell_1^2 + 2\sum_{2\leq i<j\leq n}\ell_i\ell_j \cos(\theta_i-\theta_j)+2t\sum_{i=1}^n \ell_i \cos\theta_i,
\end{eqnarray*}
where we set $\theta_1=\theta_n=0$. In particular, we can write this as
\begin{eqnarray}\label{ffunction}
 f(\theta,t)&=&C_\ell + f_1(\theta) + tf_2(\theta)
\end{eqnarray}
where $\theta=(\theta_2,\ldots,\theta_{n-1})$ and $C_\ell=\|\ell\|^2-2\ell_1^2$. Taking partial derivatives, we get
\begin{eqnarray}\label{origderi}
 \frac{\partial f}{\partial \theta_k}&=& 2\ell_k\left(\sum_{i=1}^n \tilde{\ell}_i\sin(\theta_i-\theta_k)\right)
\end{eqnarray}
for $k=2,\ldots,n-1$, where $\tilde{\ell}_1=t$ and $\tilde{\ell}_i=\ell_i$ for $i\geq 2$, and
\begin{eqnarray*}
 \frac{\partial f}{\partial t}&=&2\sum_{i=1}^n\ell_i \cos\theta_i.
\end{eqnarray*}
Assume that $\frac{\partial f}{\partial \theta_k}=0$ for all $k=2,\ldots,n-1$. Then
\begin{eqnarray}\label{cossinstuff}
 \cos \theta_k\sum_{i=1}^n\tilde{\ell}_i\sin\theta_i&=&\sin\theta_k\sum_{i=1}^n\tilde{\ell}_i\cos\theta_i
\end{eqnarray}
for all $k=2,\ldots,n-1$. If $\sum \tilde{\ell}_i\cos\theta_i=0$, we have $\frac{\partial f}{\partial t}=2(\ell_1-t)\not=0$, so we do not get a critical point. Therefore we can assume that $\sum \tilde{\ell}_i\cos\theta_i\not=0$.

If $\sum \tilde{\ell}_i\sin\theta_i=0$, we get from (\ref{cossinstuff}) that $\sin\theta_k=0$ for all $k=2,\ldots,n-1$. Therefore we have a collinear configuration, with $\theta_k=0$ or $\pi$ for $k=1,\ldots,n$. But then $\frac{\partial f}{\partial t}\not=0$ as $\ell$ is generic.

If $\sum\tilde{\ell}_i\sin\theta_i\not=0$, then $\tan \theta_k$ is defined and nonzero, and it is independent of $k$. Hence the $\theta_2,\ldots,\theta_{n-1}$ are collinear. But by (\ref{origderi}) we get $\tilde{\ell_1}\sin(-\theta_k)+\tilde{\ell}_n\sin(-\theta_k)=0$, a contradiction.
\end{proof}

Thus let
\begin{eqnarray*}
 W&=&f^{-1}(\{0\}),
\end{eqnarray*}
a compact cobordism. Note that if $(z,t)\in W$, we get
\begin{eqnarray*}
\left\|\sum_{i=2}^{n-1} \ell_i z_i+(\ell_n+t) \right\|^2& = & (\ell_1-t)^2,
\end{eqnarray*}
so $((-\sum_{i=2}^{n-1}\ell_iz_i-\ell_n)/(\ell_1-t),z_2,\ldots,z_{n-1})\in \mathcal{M}_{\ell_t}$ and all elements of $\mathcal{M}_{\ell_t}$ arise this way. Therefore $W$ is a cobordism between $\mathcal{M}_\ell$ and $\mathcal{M}_{\ell_{\ell_1-\varepsilon}}$. We also get a natural inclusion $i:W\to T^{n-1}\times [0,\ell_1-\varepsilon]$ extending the inclusions $\mathcal{M}_{\ell_t}\to T^{n-1}$.

Since $\ell$ is generic, we can assume that $\ell_n>\ell_1$.

\begin{lemma}\label{morsefctg}
 The map $g:W\to [0,\ell_1-\varepsilon]$ given by $g(z,t)=t$ is a Morse function. Its critical points are in one-to-one correspondence with those subsets $J\subset \{2,\ldots,n-1\}$ which are short with respect to $\ell$ and such that $\{1\}\cup J$ is long with respect to $\ell$. The critical point corresponding to such $J$ is given by
\begin{eqnarray*}
q_J&=&(u_2,\ldots,u_{n-1},t)
\end{eqnarray*}
where $u_j=1$ for $j\notin J$ and $u_j=-1$ for $j\in J$ and $t$ satisfies
\begin{eqnarray*}
 2t\left(\sum_{j=1}^nu_j\ell_j\right)&=&\ell_1^2-\left(\sum_{j=2}^nu_j\ell_j\right)^2.
\end{eqnarray*}
Furthermore, the index of $q_J$ is $|J|$.
\end{lemma}

It is straightforward to check that $f(q_J)=0$. Also, the conditions that $\{1\}\cup J$ is long while $J$ is short with respect to $\ell$ ensures that
\[
 \sum_{j=1}^nu_j\ell_j>0\mbox{ and }\sum_{j=2}^nu_j \ell_j \in(-\ell_1,\ell_1)
\]
so that $q_J$ is an interior point of $W$. Note that this gives a more precise condition on $\varepsilon>0$: it is so small that $\sum_{j=2}^nu_j\ell_j<\ell_1-\varepsilon$ for all $J\subset \{2,\ldots,n-1\}$ which are short with respect to $\ell'$ and such that $\{1\}\cup J$ is long with respect to $\ell$. Again $u_j=1$ for $j\notin J$ and $u_j=-1$ for $j\in J$.

\begin{proof}
If $\frac{\partial f}{\partial \theta_k}\not=0$ at $(\theta_2,\ldots,\theta_{n-1},t)\in W\subset T^{n-2}\times [0,\ell_1-\varepsilon]$ for some $k\in\{2,\ldots,n-1\}$, then the $k$-th coordinate of a neighborhood of this point in $W$ is determined by the other coordinates, including the $t$-coordinate, by the implicit function theorem. Therefore projection to the $t$-coordinate does not contain critical points in this neighborhood. Hence to get a critical point of $g$, we need $\frac{\partial f}{\partial \theta_k}=0$ for all $k=2,\ldots,n-1$.

Recall that then equation (\ref{cossinstuff}) holds with $\theta_1=\theta_n=0$. If $\sum\tilde{\ell}_i\cos\theta_i=0$ and $\sum\tilde{\ell}_i\sin\theta_i=0$, we get a closed configuration $(u_1,\ldots,u_n)$ for $\tilde{\ell}$ satisfying $u_1=u_n=1$. But then $f(u_2,\ldots,u_{n-1},t)=-(\ell_1-t)^2\not=0$ for all $t\in [0,\ell_1-\varepsilon]$.

We can therefore assume that $\sum\tilde{\ell}_i\cos\theta_i\not=0$ or $\sum\tilde{\ell}_i\sin\theta_i\not=0$. If $\sum\tilde{\ell}_i\cos\theta_i=0$, we get from (\ref{cossinstuff}) that $\cos\theta_k=0$ for $k=2,\ldots,n-1$. Then
\begin{eqnarray*}
f(\theta_2,\ldots,\theta_{n-1},t)&\geq &(\ell_n+t)^2-(\ell_1-t)^2 \,\,\,>\,\,\,0.
\end{eqnarray*}
Recall that we assume $\ell_n>\ell_1$. But these configurations do not lie in $W$.

If $\sum\tilde{\ell}_i\sin\theta_i=0$, we can argue as in the proof of Lemma \ref{critpointsf} to get a collinear configuration $(u_2,\ldots,u_{n-1},t)$ with $u_i\in \{\pm 1\}$. It is easy to check that the condition $f(u_2,\ldots,u_{n-1},t)=0$ and $t\in[0,\ell_1-\varepsilon]$ leads exactly to the points $q_J$ where $J\subset \{2,\ldots,n-1\}$ which are short with respect to $\ell'$ and such that $\{1\}\cup J$ is long with respect to $\ell$.

For these points we know from Lemma \ref{critpointsf} that $\frac{\partial f}{\partial t}\not=0$, so we can express the $t$-coordinate in terms of $(\theta_2,\ldots,\theta_{n-1})$ near $q_J$. It follows from (\ref{ffunction}) that in these coordinates, with $\theta=(\theta_2,\ldots,\theta_{n-1})$ we have
\begin{eqnarray*}
 g(\theta)&=&\frac{-(C_\ell+f_1(\theta))}{f_2(\theta)}
\end{eqnarray*}
with $f_1$ and $f_2$ as defined by (\ref{ffunction}), and hence
\begin{eqnarray*}
 \frac{\partial g}{\partial{\theta_k}}&=&\frac{(C_\ell+f_1(\theta))\frac{\partial f_2}{\partial \theta_k}}{(f_2(\theta))^2} - \frac{\frac{\partial f_1}{\partial \theta_k}}{f_2(\theta)}.
\end{eqnarray*}
As the partial derivatives of $f_1$ and $f_2$ are sums of $\sin \theta_k$, we get that the $q_J$ are indeed critical points of $g$. Furthermore, taking further partial derivatives, we get that
\begin{eqnarray*}
 \frac{\partial^2 g}{\partial \theta_k\partial \theta_l}(q_J) &=& \frac{-1}{f_2(u_2,\ldots,u_{n-1})}\left(t\frac{\partial^2 f_2}{\partial \theta_k\partial \theta_l} +\frac{\partial^2 f_1}{\partial \theta_k\partial \theta_l}\right).
\end{eqnarray*}
Note that $f_2(u_2,\ldots,u_{n-1})=\sum_{j=1}^n \ell_j u_j>\ell_1$ by the conditions on $J$ so the factor in front is negative.

Taking partial derivatives of $f_1$ and $f_2$ gives
\begin{eqnarray*}
 -f_2(u_2,\ldots,u_{n-1})\frac{\partial^2 g}{\partial \theta_k\partial \theta_l}(q_J) &=& -2t\Delta(\ell_2 u_2,\ldots,\ell_{n-1}u_{n-1})+\left(2\ell_k\ell_l u_ku_l\right)_{k,l=2}^{n-1}
\end{eqnarray*}
where $\Delta$ is a diagonal matrix with the appropriate diagonal terms. The index of $q_J$ can now be calculated analogously to the calculation in the proof of \cite[Lm.1.4]{faritr}, and it is indeed the cardinality of $J$.

We need the signature of $\Delta(t\ell_2u_2,\ldots,t\ell_{n-1}u_{n-1})-(\ell_ku_k\ell_mu_m)_{k,m=2}^{n-1}$, which is congruent to $D-E$, where $D=\Delta(\frac{t}{\ell_2}u_2,\ldots,\frac{t}{\ell_{n-1}}u_{n-1})$ and $E$ has 1 in every entry.

For simplicity, reorder the elements $\{2,\ldots,n-1\}$ so that $u_2,\ldots,u_i=-1$ and $u_{i+1},\ldots,u_{n-1}=1$. In particular, we have $|J|=i-1$. Let $D_r-E_r$ be the principal minor of size $r$ of $D-E$. Then by \cite[Lm.1.5]{faritr} we have
\begin{eqnarray*}
 \det(D_r-E_r) &=& \prod_{j=2}^r\frac{t}{\ell_j}u_j\left(1-\sum_{j=2}^r \frac{\ell_j}{t}u_j\right)\\
&=&t^{r-1}\prod_{j=2}^r\frac{u_j}{\ell_j}\left(t-\sum_{j=2}^r\ell_ju_j\right)
\end{eqnarray*}
Now $t-\sum_{j=2}^r\ell_ju_j=t+\sum_{j=2}^i \ell_j-\sum_{j=i+1}^r \ell_j$ which is positive for all $r\leq n-1$, because
\begin{eqnarray*}
 \ell_1+\sum_{j=2}^i \ell_j &>& \sum_{j=i+1}^{n-1}\ell_j +\ell_n
\end{eqnarray*}
as $J\cup\{1\}$ is a long subset. Therefore, the sign of the determinant of $D_r-E_r$ changes exactly $(i-1)$-times, and by the Sylvester criterion this implies that the signature is $i-1=|J|$.
\end{proof}

\begin{lemma}\label{subsetssmall}
 We have $\dot{\mathcal{S}}(\ell_{\ell_1-\varepsilon})\subset \dot{\mathcal{S}}(\ell)$, and $\dot{\mathcal{S}}(\ell)-\dot{\mathcal{S}}(\ell_{\ell_1-\varepsilon})$ is in 1-1 correspondence with the critical points of $g:W\to \R$.
\end{lemma}

\begin{proof}
 Let $J\in \dot{\mathcal{S}}(\ell_{\ell_1-\varepsilon})$. Assume $1\in J$. Then
\begin{eqnarray*}
 \varepsilon+\sum_{j\in J-\{1\}}\ell_j+(\ell_n+\ell_1-\varepsilon)&<&\sum_{j\notin J\cup\{n\}}\ell_j
\end{eqnarray*}
which implies that $J\cup\{n\}$ is short with respect to $\ell$, that is, $J\in\dot{\mathcal{S}}(\ell)$.

If $1\notin J$, then
\begin{eqnarray*}
 \sum_{j\in J} \ell_j+(\ell_n+\ell_1-\varepsilon)&<& \sum_{j\notin J\cup\{1,n\}}\ell_j+\varepsilon.
\end{eqnarray*}
Since we can assume that $\ell_1>2\varepsilon$, we get $J\in\dot{\mathcal{S}}(\ell)$.

Now if $J\in \dot{\mathcal{S}}(\ell)-\dot{\mathcal{S}}(\ell_{\ell_1-\varepsilon})$, it is easy to see from the argument above that $1\notin J$. Let $K=\{2,\ldots,n-1\}-J$. The condition $J\notin \dot{\mathcal{S}}(\ell_{\ell_1-\varepsilon})$ gives
\begin{eqnarray*}
 \sum_{j\in J}\ell_j+\ell_n+\ell_1 & > & \sum_{j\in K}\ell_j
\end{eqnarray*}
so $K$ is short with respect to $\ell$. But since $J\in \dot{\mathcal{S}}(\ell)$, we get $K\cup \{1\}$ is long with respect to $\ell$. The reverse argument shows that every such $K$ corresponds to an element of $\dot{\mathcal{S}}(\ell)-\dot{\mathcal{S}}(\ell_{\ell_1-\varepsilon})$.
\end{proof}

Let us analyze potential indices for critical points. To get a critical point of index 0, we would need the set $\{1\}$ to be long which is impossible. To get a critical point of index 1, we need $\{1,i\}$ to be long for some $i\in \{2,\ldots,n-1\}$. In particular we get $\{1,n\}$ to be long and $\mathcal{M}_\ell$ is a $(n-3)$-sphere.

To get a critical point of index 2, we need $\{1,n-2,n-1\}$ to be long. In particular, if $\ell$ is special, it is of type $\{1,n-2,n-1\}$, a cohomologically rigid type. By standard transversality arguments we now get

\begin{corollary}\label{fundgrptrns}
 Let $\ell$ be a special length vector of type different from $\{1,n-2,n-1\}$. Then $\mathcal{M}_\ell$ has the same fundamental group as the cobordism $W$.
\end{corollary}

To understand the fundamental group of $W$, we start with
\begin{eqnarray}\label{prodfund}
\pi_1(\mathcal{M}_{\ell_{\ell_1-\varepsilon}})& \cong & \pi_1(\mathcal{M}_{\ell'})\times \Z,
\end{eqnarray}
where $\ell'=(\ell_2,\ldots,\ell_{n-1},\ell_n+\ell_1)$ and analyze the critical points of index $n-3$ and $n-4$. Note that there can only be a critical point of index $n-2$ if $\{n\}$ is long with respect to $\ell'$.

Before we consider the critical points of index $n-3$ and $n-4$ let us first observe that the types of $\ell$ and $\ell'$ are closely related.

\begin{lemma}\label{typeinduction}
 Let $\ell$ be a special length vector and $\ell'=(\ell_2,\ldots,\ell_{n-1},\ell_n+\ell_1)$.
\begin{enumerate}
 \item If $\ell$ is of type $\{i,n-2,n-1\}$ with $i\geq 2$, then $\ell'$ is special of type $\{i-1,(n-1)-2,(n-1)-1\}$.
\item For $n\geq 6$, if $\ell$ is of type $\{n-4,n-3,n-1\}$, then $\ell'$ is special of type $\{(n-1)-4,(n-1)-3,(n-1)-1\}$.
\end{enumerate}

\end{lemma}

\begin{proof}
 Note that $J\subset \{2,\ldots,n-1\}$ is short with respect to $\ell$ if and only if $J$ is short with respect to $\ell'$. The claim now follows easily.
\end{proof}

Lemma \ref{typeinduction} will allow us to determine the fundamental group of a special length vector via induction. We now want to set up basic objects for the fundamental group. The following Lemma is a useful tool for constructing actual loops in $\mathcal{M}_\ell$.

\begin{lemma}\label{longarm}
 Let $\ell=(\ell_1,\ldots,\ell_m)$ be an ordered length vector with $m\geq 2$ and let $p:T^m\to \R^2$ be given by
\begin{eqnarray*}
 p(z_1,\ldots,z_m)&=&\sum_{j=1}^m \ell_j z_j.
\end{eqnarray*}
Then there exists a path $\gamma:[\ell_m-\sum_{j=1}^{m-1} \ell_j,\ell_m+\sum_{j=1}^{m-1}\ell_j]\to T^m$ such that
$p\circ\gamma(t)=(t,0)$ for all $t$, and if $q_j:T^m\to S^1$ denotes projection to the $j$-th coordinate, we have
\begin{enumerate}
 \item $q_j\circ\gamma(t)\notin\{-1,1\}$ for at most two values of $j\in\{1,\ldots,m\}$.
\item ${\rm Im}(q_j\circ \gamma(t))\geq 0$ for $j\leq n-1$ and all $t$.
\item ${\rm Im}(q_m\circ \gamma(t))\leq 0$ for all $t$.
\end{enumerate}
Here ${\rm Im}(z)$ is the imaginary part of the complex number $z\in S^1\subset \C$.
\end{lemma}

\begin{proof}
 For $n=2$ consider the map $\tilde{\gamma}(t)=(e^{it},1)$ for $t\in [0,\pi]$. Then ${\rm Im}(p\tilde{\gamma}(t))\geq 0$ for all $t$. So for every $t$ we can find an angle $\theta(t)\in[-\pi/2,0]$ such that $p(e^{it},e^{i\theta(t)})\in\R$. Reparametrizing this path gives $\gamma$ for $n=2$. Note that $q_2(\gamma(t))=1$ at the endpoints of the interval, while $q_1(\gamma(t))$ starts with $-1$ and ends with $1$. For $n>2$ we iterate this construction.
\end{proof}

Assume that $\ell$ is an ordered special length vector, so $\{n-3,n-2,n-1\}$ is a long set. We begin by choosing a basepoint in $\mathcal{M}_\ell$. Define $z^\ast=(z_1^\ast,\ldots,z_{n-1}^\ast)$ by $z_j=1$ for $j\in \{1,\ldots,n-4\}$, $z^\ast_{n-3}=-1$ and $z^\ast_{n-2}$, $z^\ast_{n-1}$ are chosen to give a closed linkage ($z_n=1$ here). This is possible, as $\{n-3,n-2,n-1\}$ is long. In fact, we have two choices for $z^\ast_{n-2}$ and $z^\ast_{n-1}$ so let us pick the one with ${\rm Im}(z^\ast_{n-2})>0$.

We will now assume that $\ell$ is of type $\{i,n-2,n-1\}$ with $i\leq n-4$. For $j\leq i-1$, define $\gamma_j:(S^1,1)\to \mathcal{M}_\ell,z^\ast)$ by $q_j(\gamma_j(z))=z$, $q_k(\gamma_j(z))=1$ for $k\in\{1,\ldots,n-4\}-\{j\}$, and the linkage is closed by using Lemma \ref{longarm} with $(\ell_{n-3},\ell_{n-2},\ell_{n-1})$ together with a small rotation, compare Figure \ref{closedlink}. Here $q_k:\mathcal{M}_\ell\to S^1$ is projection to the $k$-th coordinate.

\begin{figure}[h]
\begin{center}
\includegraphics[height=2.5cm,width=8cm]{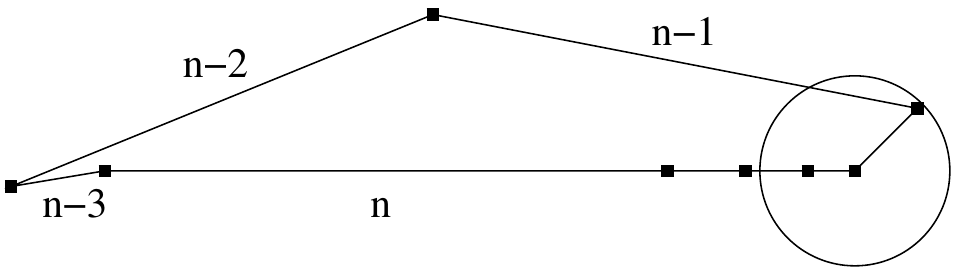}
\caption{\label{closedlink} }
\end{center}
\end{figure}

The fundamental group of $T^{n-1}$ is $\Z^{n-1}$ generated by $x_1,\ldots,x_{n-1}$, where each $x_j$ is generated by rotating the $j$-th coordinate once and keeping all other coordinates fixed. It is then clear that for $j\leq i-1$ we have $i_\#([\gamma_j])=x_j\in \pi_1(T^{n-1})$.

It is easy to see that $[\gamma_j]$ commutes with $[\gamma_k]$ for $j,k\leq i-1$.

For $k\in\{n-2,n-1\}$ we may have $\{k,n\}$ is long or short. If it is short, define $\gamma_k:(S^1,1)\to (\mathcal{M}_\ell,z^\ast)$ by rotating the $k$-th coordinate. Using Lemma \ref{longarm}, this can be done so that $i_\#([\gamma_k])=x_k\in \pi_1(T^{n-1})$. In fact, if $j\leq i-1$ is such that $\{j,k,n\}$ is short, this can be done by keeping the $j$-th coordinate fixed. Note that $\{i,k,n\}$ is already long by the type restriction. It is then not difficult to see that $[\gamma_k]$ commutes with $[\gamma_j]$, and in fact with every $[\gamma_l]$ for $l\leq j$.

To generate the fundamental group, we need to assign elements for $j\in \{i,\ldots,n-3\}$. We will do this in the next section for $i=1$, and use Lemma \ref{typeinduction} and (\ref{prodfund}) to obtain the general case.

\section{Cohomology of special length vectors of type $\{1,n-2,n-1\}$}\label{seci=1}

We have already seen in Proposition \ref{lamerigid} that type $\{1,n-2,n-1\}$ is cohomologically rigid. However, knowing the cohomology and the fundamental group for these polygon spaces serves as an induction start for the types $\{i,n-2,n-1\}$ with $i\geq 2$.

Let $\ell=(1,\ldots,1,n-3,n-3,n-2)\in \R^n$. Simple checking shows that $\ell$ is of type $\{1,n-2,n-1\}$ and that $\{n-2,n\}$ is long. On the other hand, $\ell'=(1,\ldots,1,n-2,n-2,n-2)\in \R^n$ has disconnected planar polygon space, as $\{n-2,n-1\}$ is long with respect to $\ell'$. The collection of length vectors $\ell_t=(1,\ldots,1,n-3+t,n-3+t,n-2)$ provides a cobordism between $\mathcal{M}_\ell$ and $\mathcal{M}_{\ell'}\cong T^{n-3}\sqcup T^{n-3}$. Using Morse theory as in Section \ref{diffeosec}, we get that $\mathcal{M}_\ell\cong T^{n-3}\# T^{n-3}$, a result which has been obtained previously by Hausmann \cite[Ex.2.10]{hausma}.

Let us give generators for the fundamental group in the case $n\geq 5$. For $\ell'$ we have two basepoints, $z^\ast$ from the last section and its complex conjugate $\bar{z}^\ast$. It is easy to see that $z^\ast$ and $\bar{z}^\ast$ are in different components. By the discussion above, we get
\[
\pi_1(\mathcal{M}_{\ell'},z^\ast)\,\,\,\cong \,\,\,\Z^{n-3}\,\,\,\cong\,\,\, \pi_1(\mathcal{M}_{\ell'},\bar{z}^\ast),
\]
and we can represent generators by $\delta_j:(S^1,1)\to (\mathcal{M}_{\ell'},z^\ast)$ and $\bar{\delta}_j:(S^1,1)\to (\mathcal{M}_{\ell'},\bar{z}^\ast)$ for $j=1,\ldots,n-3$, which rotate the $j$-th coordinate, and keep the $k$-th coordinate fixed for $k\in\{1,\ldots,n-3\}-\{j\}$. If we use $\bar{\delta}_j$ as the complex conjugate of $\delta_j$, we get $a_j=[\delta_j]\in \pi_1(\mathcal{M}_{\ell'},z^\ast)$ and $b_j=[\bar{\delta}_j]\in \pi_1(\mathcal{M}_{\ell'},\bar{z}^\ast)$ with $i_\#(a_j)=x_j\in \pi_1(T^{n-1})$ and $i_\#(b_j)=-x_j\in\pi_1(T^{n-1})$ for all $j\in \{1,\ldots,n-3\}$.

Now in $\mathcal{M}_\ell$ we can find a path $\delta$ between $z^\ast$ and $\bar{z}^\ast$, so we can obtain elements $b_j\in \pi_1(\mathcal{M}_\ell$ by $b_j=[\delta\ast \bar{\delta}_j\ast \delta^{-1}]$ for $j=1,\ldots,n-3$. Note however that we have to slightly change $\delta_{n-3}$, namely, we cannot just fix all coordinates $k\leq n-4$, as $\{1,\ldots,n-3,n\}$ is not a short set with respect to $\ell$. But when we rotate the $n-3$-rd coordinate, we change the $n-4$-th coordinate so that $z_{n-2}$ and $z_{n-1}$ can close the linkage. The easiest way would be $z_{n-4}=-z_{n-3}$, but then the $n-4$-th coordinate makes a full turn. It is possible to avoid a full turn by choosing $z_{n-4}=-z_{n-3}$ until $z_{n-4}=-1$ and then rotate backwards, compare Figure \ref{linkcircle}, where the second bar does a full turn while the first bar only does a half turn and then goes back at a different speed. Then $\pi_1(\mathcal{M}_\ell,z^\ast)$ is generated by $a_1,\ldots,a_{n-3},b_1,\ldots,b_{n-3}$.

\begin{figure}[h]
\begin{center}
\includegraphics[height=2cm,width=10cm]{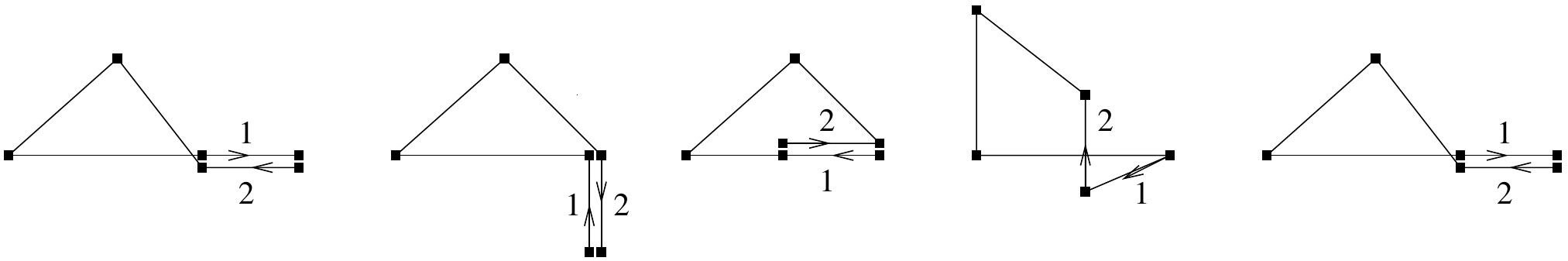}
\caption{\label{linkcircle} }
\end{center}
\end{figure}

Also, if $n\geq 6$ we get
\begin{eqnarray*}
 \pi_1(\mathcal{M}_\ell,z^\ast)&\cong & \Z^{n-3} \ast \Z^{n-3}
\end{eqnarray*}
with the first factor generated by $a_1,\ldots,a_{n-3}$ and the second factor generated by $b_1,\ldots,b_{n-3}$. If $n=5$, we get $T^2\# T^2$ is a surface of genus 2, and
\begin{eqnarray}\label{surfacegrp}
 \pi_1(\mathcal{M}_\ell,z^\ast)&\cong & \langle a_1,a_2,b_1,b_2\,|\,[a_1,a_2][b_1,b_2] \,\rangle.
\end{eqnarray}

If $n\geq 6$ it is possible that $\{k,n\}$ is short for $k\in\{n-2,n-1\}$. In this case a simple Morse theory argument, compare also \cite[Ex.2.11]{hausma}, gives
\begin{eqnarray}\label{fundgrpi=1}
 \pi_1(\mathcal{M}_\ell,z^\ast)&\cong & \Z^{n-3} \ast \Z^{n-3} \ast F_l
\end{eqnarray}
with $l$ the number of elements in $\{n-2,n-1\}$ with $\{k,n\}$ short.

\begin{proposition}\label{cohlametype}
 Let $\ell$ be a special length vector of type $\{1,n-2,n-1\}$ with $n\geq 5$. Then
\begin{eqnarray*}
 H^\ast_{(1)}(\mathcal{M}_\ell;\Z)&\cong & \Lambda_\Z[A_1,\ldots,A_{n-1},B_1,\ldots,B_{n-3}]/I,
\end{eqnarray*}
where $I$ is the ideal generated by
\begin{gather*}
 A_1\cdots A_{n-3}+ (-1)^{n-3} B_1\cdots B_{n-3} \\
 A_iB_k\mbox{ for } i\in\{1,\ldots,n-1\}, k\in\{1,\ldots,n-3\} \\
 A_iA_k \mbox{ for } i\in\{1,\ldots,n-1\}, k\in\{n-2,n-1\}\\
 A_k  \mbox{ if } \{k,n\}\mbox{ is long.}
\end{gather*}

\end{proposition}

\begin{proof}
 For $i\in\{1,\ldots,n-1\}$ define $A_i\in H^1(\mathcal{M}_\ell;\Z)\cong {\rm Hom}(\pi_1(\mathcal{M}_\ell),\Z)$ by $A_i(a_i)=1$ and $A_i$ vanishes on the other generators. Here $A_{n-2}$ and $A_{n-1}$ are only defined if $\{n-2,n\}$, respectively $\{n-1,n\}$, are short. Similarly we define $B_i\in H^1(\mathcal{M}_\ell;\Z)$ for $i\in\{1,\ldots,n-3\}$. Then $H^\ast_{(1)}(\mathcal{M}_\ell;\Z)$ is generated by $A_1,\ldots,A_{n-1},B_1,\ldots,B_{n-3}$, and all relations except the first one hold by the fundamental group discussion above.

To see that the first relation holds, note that
\begin{eqnarray*}
 H^\ast(T^{n-1};\Z)&\cong & \Lambda_\Z[X_1,\ldots,X_{n-1}]
\end{eqnarray*}
where $X_i\in H^1(T^{n-1};\Z)$ satisfies $X_i(x_j)=1$ for $i=j$ and 0 for $i\not=0$. Since $i_\#(a_j)=x_j$ and $i_\#(b_j)=-x_j$, we get $i^\ast(X_j)=A_j-B_j$ for $j=1,\ldots,n-3$, and $i^\ast(X_k)=A_k$ for $k\in \{n-2,n-1\}$. By Theorem \ref{thefacering} we get
\begin{eqnarray*}
 0&=&i^\ast(X_1\cdots X_{n-3})\\
&=&(A_1-B_1)\cdots (A_{n-3}-B_{n-3})\\
&=&A_1\cdots A_{n-3}+ (-1)^{n-3} B_1\cdots B_{n-3}
\end{eqnarray*}
since all mixed terms are zero. This shows the first relation. There cannot be any other relations for otherwise we get wrong Betti numbers. Here note that $H^k_{(1)}(\mathcal{M}_\ell;\Z)=H^k(\mathcal{M}_\ell;\Z)$, except possibly for $k=n-4$, where the full cohomology also contains the Poincar\'e duals of $A_{n-2}$ and $A_{n-1}$.
\end{proof}

\section{Cohomology of special length vectors of type $\{i,n-2,n-1\}$ for $i\leq n-4$}\label{sectypen-4}

Before we determine the cohomology, we first determine the fundamental group. For this we will distinguish the cases $i\leq n-5$ and $i=n-4$. The main difference is that for $i\leq n-5$ we automatically have that $\{n-4,n-3,n\}$ is short, while for $i=n-4$ this may be short or long.

Let $\ell$ be special of type $\{i,n-2,n-1\}$ with $i\leq n-5$. We define a graph $\Gamma_\ell$ as follows. The vertex set is $\{a_1,\ldots,a_k,b_i,\ldots,b_{n-3}\}$ with $k=|\dot{\mathcal{S}}_1(\ell)|$. Note that $k\geq n-3$. The edges are given by
\begin{eqnarray*}
 \{a_j,a_m\}&\mbox{for} & 1\leq j<m\leq n-3 \\
\{a_j,b_m\}& \mbox{for} & 1\leq j\leq i-1,\,i\leq m \leq n-3 \\
\{b_j,b_m\}& \mbox{for} & i\leq j<m<n-3 \\
\{a_j,a_m\}& \mbox{for} & 1\leq j \leq i-1, \, m\geq n-2, \mbox{ provided that }\{j,m,n\} \mbox{ is short.}
\end{eqnarray*}

Now let $G_\ell$ be the right-angled Artin group corresponding to $\Gamma_\ell$, that is, $G_\ell$ is generated by the vertices of $\Gamma_\ell$, two generators commute if and only if they span an edge, and these are the only relations among the generators.

Note that if $k=n-3$, we get $G_\ell\cong \Z^{i-1}\times (\Z^{n-2-i}\ast \Z^{n-2-i})$.

\begin{proposition}\label{fundgrpin-1}
 Let $n\geq 6$ and $\ell$ a special length vector of type $\{i,n-2,n-1\}$ with $i\leq n-5$. Then
\begin{eqnarray*}
 \pi_1(\mathcal{M}_\ell)&\cong & G_\ell.
\end{eqnarray*}
Furthermore, the homomorphism $\varphi:G_\ell\to \Z^{n-1}$ induced by the inclusion $i:\mathcal{M}_\ell\to T^{n-1}$ sends each $a_j$ to $x_j$ and $b_j$ to $-x_j$, where each $x_j$ is represented by rotating the $j$-th coordinate of $T^{n-1}$, $j=1,\ldots,n-1$.
\end{proposition}

\begin{proof}
 The proof is by induction on $i$. For $i=1$ this is (\ref{fundgrpi=1}), the statement about $\varphi$ follows from the discussion in Section \ref{seci=1}.

So now assume that $i\geq 2$. We have the cobordism $W$ between $\mathcal{M}_\ell$ and $\mathcal{M}_{\ell_{\ell_1-\varepsilon}}$ from Lemma \ref{morsefctg}. By Proposition \ref{smallfirst} we get
\begin{eqnarray*}
\pi_1(\mathcal{M}_{\ell_{\ell_1-\varepsilon}})&\cong &\Z\times \pi_1(\mathcal{M}_{\ell'})
\end{eqnarray*}
with $\ell'=(\ell_2,\ldots,\ell_{n-1},\ell_n+\ell_1)$. Since $\ell'$ is of type $\{i-1,n-3,n-2\}$, we get $\pi_1(\mathcal{M}_{\ell'})\cong G_{\ell'}$, generated by $a'_1,\ldots,a'_{k'},b'_{i-1},\ldots,b'_{n-4}$ and with corresponding $\varphi'$. Defining $a_1=[\gamma_1]\in\pi_1(\mathcal{M}_{\ell_{\ell_1-\varepsilon}})$, $a_j=a'_{j-1}$ for $j\in\{2,\ldots,k'+1\}$ and $b_j=b'_{j-1}$ for $j\in\{i,\ldots,n-3\}$, it is easy to see that the statement of Proposition \ref{fundgrpin-1} holds for $\ell_{\ell_1-\varepsilon}$.

The Morse function $g$ on $W$ can have at most two critical points of index $n-3$ in which case $\pi_1(W)$ has up to two extra generators ($a_{n-2}$ and $a_{n-1}$), represented by the appropriate $\gamma_j$ defined in Section \ref{diffeosec}. Note that the projection map $\mathcal{M}_{\ell_{\ell_1-\varepsilon}}\to S^1$ is not surjective for these coordinates, so they do not represent elements of $\pi_1(\mathcal{M}_{\ell_{\ell_1-\varepsilon}})$. Also, the map $i_\ast:\pi_1(W)\to\pi_1(T^{n-1})$ has the right properties\footnote{Recall that we have an inclusion $W\to T^{n-1}\times [0,\ell_1-\varepsilon]$ extending the inclusions $\mathcal{M}_{\ell_t}\to T^{n-1}$.}.

If $g$ has a critical point $q_J$ of index $n-4$, we get a subset $\{j,m\}=\{2,\ldots,n-1\}-J$ with $\{j,m,n\}$ short with respect to $\ell$, but long with respect to $\ell'$. Since $\ell$ is of type $\{i,n-2,n-1\}$ and $\ell'$ of type $\{i-1,n-3,n-2\}$, it is not possible that both $j$ and $m$ are less than $n-2$. Assume $j<m$ and $m\geq n-2$. Then $j<i$ for otherwise $\{j,m,n\}$ is long with respect to $\ell$. Note that $a_j,a_m$ do not commute in $\pi_1(\mathcal{M}_{\ell_t})$ for $t>g(q_J)$, as the projection $\mathcal{M}_{\ell_t}\to S^1\times S^1$ to the coordinates $j$ and $m$, is not surjective (the point $(1,1)$ is not in the image as $\{j,m,n\}$ is long with respect to $\ell_t$). However for $t\leq g(q_J)$ they do commute. Also note that $q_j=(u_2,\ldots,u_{n-1},t_0)$ with $u_s=-1$ except for $s=j,m$, where $u_s=1$. We can embed a 2-torus into $\mathcal{M}_{\ell_{t_0}}$ spanned by $\gamma_j$ and $\gamma_m$, and the point $(1,1)\in T^2$ is mapped to $q_J$. By the Seifert-Van Kampen theorem together with Morse theory we see that $\pi_1(\mathcal{M}_{\ell_{t_0}})$ is $\pi_1(\mathcal{M}_{\ell_t})$ with a commutator relation $[a_j,a_m]$ added (for small values $t>t_0$).

Since critical points of lower index have no impact on $\pi_1(W)$, we get $\pi_1(W)\cong G_\ell$, and the result follows from Corollary \ref{fundgrptrns}.
\end{proof}

Now let $\ell$ be a special length vector of type $\{n-4,n-2,n-1\}$ with $n\geq 5$. Let us first assume that $\{n-3,n-4,n\}$ is long. Then let $\bar{G}_\ell$ be the group generated by $a_1,\ldots,a_k,b_{n-4},b_{n-3}$ with $k=|\dot{\mathcal{S}}_1(\ell)|$, subject to the relations 
\begin{gather*}
[a_{n-4},a_{n-3}][b_{n-4},b_{n-3}]\\
[a_j,a_m]\mbox{ for }\{j,m,n\}\mbox{ short, and }\\
[a_j,b_m]\mbox{ for }j\leq n-5, m\in\{n-4,n-3\}.
\end{gather*}
Note that if $k=n-3$, we get $\bar{G}_\ell\cong \Z^{n-5}\times S_2$, where $S_2$ is the fundamental group of the orientable surface of genus 2.

\begin{proposition}\label{fundgrpn-4long}
 Let $\ell$ be a special length vector of type $\{n-4,n-2,n-1\}$ with $n\geq 5$ and $\{n-3,n-4,n\}$ is long. Then
\begin{eqnarray*}
 \pi_1(\mathcal{M}_\ell)&\cong & \bar{G}_\ell
\end{eqnarray*}
Furthermore, the homomorphism $\varphi:G_\ell\to \Z^{n-1}$ induced by the inclusion $i:\mathcal{M}_\ell\to T^{n-1}$ sends each $a_j$ to $x_j$ and $b_j$ to $-x_j$, where each $x_j$ is represented by rotating the $j$-th coordinate of $T^{n-1}$, $j=1,\ldots,n-1$.
\end{proposition}

\begin{proof}
 The proof is given by induction on $n$. For $n=5$, the result follows from (\ref{surfacegrp}) and the discussion in Section \ref{seci=1}.

For $n\geq 6$ we use again the cobordism $W$ and the Morse function $g:W\to \R$. As in the proof of Proposition \ref{fundgrpin-1}, critical points of index $n-3$ produce new generators $a_{n-2}$ or $a_{n-1}$, and critical points of index $n-4$ add relations $[a_j,a_m]$, but only with $j\leq n-5$ and $m\geq n-2$.This finishes the proof.
\end{proof}

It remains to consider the case where $\ell$ is of type $\{n-4,n-2,n-1\}$ with $\{n-4,n-3,n\}$ short. First observe that in this case $n\geq 7$. For if $n=5$, $\{1,2,5\}$ being short implies $\mathcal{M}_\ell$ disconnected, and if $n=6$ and $\{2,3,6\}$ is short, $\ell$ has type $\{1,4,5\}$.

So if $n\geq 7$ and $\{n-4,n-3,n\}$ is short with respect to $\ell$, but long with respect to $\ell_{\ell_1-\varepsilon}$, then $g$ has a critical point of index $n-4$ corresponding to $J=\{2,\ldots,n-1\}-\{n-4,n-3\}$. But as in the proof of Proposition \ref{fundgrpin-1} this just adds a relation $[a_{n-4},a_{n-3}]$. With this relation, the relation $[a_{n-4},a_{n-3}][b_{n-4},b_{n-3}]$ which holds for $\ell_{\ell_1-\varepsilon}$ breaks into two commutator relations, so we get a right-angled Artin group as in the case $i\leq n-5$. Now by continuing as in the proof of Proposition \ref{fundgrpin-1} we get the following result.

Let $\Gamma_\ell$ be the graph with vertex set $\{a_1,\ldots,a_k,b_{n-4},b_{n-3}\}$ with $k=|\dot{\mathcal{S}}_1(\ell)|$, and edges given by
\begin{eqnarray*}
 \{a_j,a_m\}&\mbox{for} & 1\leq j<m\leq n-3 \\
\{a_j,b_m\}& \mbox{for} & 1\leq j\leq n-5,\,m \in\{n-4,n-3\} \\
\{b_{n-4},b_{n-3}\}& \mbox{and} & \\
\{a_j,a_m\}& \mbox{for} & 1\leq j \leq n-5, \, m\geq n-2, \mbox{ provided that }\{j,m,n\} \mbox{ is short.}
\end{eqnarray*}

\begin{proposition}\label{fundgrpn-4short}
 Let $\ell$ be a special length vector of type $\{n-4,n-2,n-1\}$ with $n\geq 7$ and $\{n-3,n-4,n\}$ is short. Then
\begin{eqnarray*}
 \pi_1(\mathcal{M}_\ell)&\cong & G_\ell,
\end{eqnarray*}
where $G_\ell$ is the right-angled Artin group with respect to $\Gamma_\ell$.

Furthermore, the homomorphism $\varphi:G_\ell\to \Z^{n-1}$ induced by the inclusion $i:\mathcal{M}_\ell\to T^{n-1}$ sends each $a_j$ to $x_j$ and $b_j$ to $-x_j$, where each $x_j$ is represented by rotating the $j$-th coordinate of $T^{n-1}$, $j=1,\ldots,n-1$.
\end{proposition}

Note that $\Gamma_\ell$ is defined as in the case where $i\leq n-5$. In particular, for $k=n-3$ we get $G_\ell\cong \Z^{n-5}\times (\Z^2\ast \Z^2)=\Z^{i-1}\times (\Z^{n-2-i}\ast \Z^{n-2-i})$ with $i=n-4$.

For $J=\{k_1,\ldots,k_m\}\subset \{1,\ldots,n-1\}$ with $k_1<\ldots<k_m$ we write
\begin{eqnarray*}
 A_J&=&A_{k_1}\cdots A_{k_m},
\end{eqnarray*}
and similarly $B_J=B_{k_1}\cdots B_{k_m}$, provided that $J\subset \{i,\ldots,n-3\}$.

\begin{theorem}\label{cohtypein-4}
 Let $\ell$ be a special length vector of type $\{i,n-2,n-1\}$ with $i\leq n-4$, $n\geq 5$. Then
\begin{eqnarray}\label{cohtypei}
 H^\ast_{(1)}(\mathcal{M}_\ell;\Z)&\cong & \Lambda_\Z[A_1,\ldots,A_{n-1},B_1,\ldots,B_{n-3}]/I_\ell,
\end{eqnarray}
where $I_\ell$ is the ideal generated by
\begin{eqnarray*}
 A_J&\mbox{for} & J\cup\{n\} \mbox{ long with } \\
& &\{n-2,n-1\}\cap J\not=\emptyset \\
 A_{J-\{i,\ldots,n-3\}}(A_{\{i,\ldots,n-3\}}+(-1)^{n-2-i} B_{\{i,\ldots,n-3\}})&\mbox{for}& J\cup\{n\} \mbox{ long with }\\
& &\{n-2,n-1\}\cap J=\emptyset \\
A_j B_k&\mbox{for}& j\in\{i,\ldots,n-1\},\\
& & k\in\{i,\ldots,n-3\}.
\end{eqnarray*}
\end{theorem}

\begin{proof}
We begin with the case that $i\leq n-5$, so $\pi_1(\mathcal{M}_\ell)$ is a right-angled Artin group.

As above, we write $G_\ell=\pi_1(\mathcal{M}_\ell)$ and $\Gamma_\ell$ for the graph determining $G_\ell$. Then
\begin{eqnarray*}
 H^\ast(G_\ell;\Z)&\cong & \Lambda_\Z[F\Gamma_\ell],
\end{eqnarray*}
where $F\Gamma_\ell$ is the flag-simplicial complex spanned by the graph $\Gamma_\ell$, see \cite{chadav}. That is, the vertices $u_0,\ldots,u_m$ in $\Gamma_\ell$ span a simplex in $F\Gamma_\ell$ if and only if $u_k,u_l$ are adjacent for all $k,l\in\{0,\ldots,m\}$. In particular, this is an exterior face ring, and we write the generators as $A_1,\ldots,A_{n-1},B_i,\ldots,B_{n-3}$. Furthermore, the inclusion $i:\mathcal{M}_\ell\to T^{n-1}$ factors through $BG_\ell$, the classifying space for $G_\ell$. We therefore get a commutative diagram
\[
 \xymatrix{H^\ast(T^{n-1};\Z)\cong\Lambda_\Z[X_1,\ldots,X_{n-1}] \ar[d]^{i^\ast} \ar[rd]^{\varphi^\ast} \\
H^\ast_{(1)}(\mathcal{M}_\ell;\Z) & H^\ast(G_\ell;\Z)\cong \Lambda_\Z[F\Gamma_\ell] \ar[l]}
\]
and by Proposition \ref{fundgrpin-1} we have
\begin{eqnarray*}
\varphi^\ast(X_j)&=&A_j\,\,\,\mbox{ for } j\in \{1,\ldots,i-1,n-2,n-1\}\\
\varphi^\ast(X_j)&=&A_j-B_j\,\,\,\mbox{ for }j\in \{i,\ldots,n-3\}.
\end{eqnarray*}
 Furthermore, the map $H^1(G_\ell;\Z)\to H^1(\mathcal{M}_\ell;\Z)$ is an isomorphism, which implies that $H^\ast(G_\ell;\Z)\to H^\ast_{(1)}(\mathcal{M}_\ell;\Z)$ is surjective.

Therefore we write $A_1,\ldots,A_{n-1}, B_i,\ldots,B_{n-3}$ for the corresponding elements in $H^1(\mathcal{M}_\ell;\Z)$. By abuse of notation, we also write $X_j=i^\ast(X_j)\in H^1(\mathcal{M}_\ell;\Z)$ for $j=1,\ldots,n-1$, leading to $X_j=A_j$ for $j\in\{1,\ldots,i-1,n-2,n-1\}$ and $X_j=A_j-B_j$ for $j\in\{i,\ldots,n-3\}$.

It is clear that $A_1,\ldots,A_{n-1},B_i,\ldots,B_{n-3}$ generate $H^\ast_{(1)}(\mathcal{M}_\ell;\Z)$, and to show (\ref{cohtypei}), we have to show that $H^\ast_{(1)}(\mathcal{M}_\ell;\Z)$ is obtained by adding the relations given in the statement.

First note that $X_J=0$ for $J\cup\{n\}$ long holds, as the $X_j$ span the balanced subalgebra, compare Proposition \ref{thefacering}. If for such $J$ we have $J\cap \{n-2,n-1\}\not=\emptyset$, we get $A_J=0$, which is seen by considering the two cases $J\cap \{i,\ldots,n-3\}=\emptyset$ or not. In the first case we simply get $A_J=X_J=0$, and in the second case we have $A_jA_m=0$ for $j\in \{i,\ldots,n-3\}$, $m\geq n-2$, as this relation holds in $H^\ast(G_\ell;\Z)$. In particular this implies $A_J=0$.

If $J\cap \{n-2,n-1\}=\emptyset$, observe that necessarily $\{i,\ldots,n-3\}\subset J$ by the type restriction on $\ell$. An easy induction argument gives
\begin{eqnarray*}
X_{\{i,\ldots,n-3\}}&=&A_{\{i,\ldots,n-3\}}+(-1)^{n-2-i}B_{\{i,\ldots,n-3\}}
\end{eqnarray*}
which implies
\begin{eqnarray*}
 A_{J-\{i,\ldots,n-3\}}(A_{\{i,\ldots,n-3\}}+(-1)^{n-2-i} B_{\{i,\ldots,n-3\}})&=&X_J
\end{eqnarray*}
and the required relation. The last type of relations hold as they hold in $H^\ast(G_\ell;\Z)$.

So we only have to show that we do not have to add any more relations. We will do this by induction on $i$.

The case $i=1$ follows easily from Proposition \ref{cohlametype}. For $i\geq 2$ we use again the cobordism $W$ between $\mathcal{M}_\ell$ and $\mathcal{M}_{\ell_{\ell_1-\varepsilon}}$. Since $\mathcal{M}_{\ell_{\ell_1-\varepsilon}}\cong S^1\times \mathcal{M}_{\ell'}$, it is easy to see that the result holds for $\ell_{\ell_1-\varepsilon}$ (after shifting the indices of $A$ and $B$), as it holds for $\ell'$ by induction.

For $t\in[0,\ell_1-\varepsilon]$, let $W_t=g^{-1}([t,\infty))$. If we start with $t=\ell_1-\varepsilon$ and then let $t$ decrease, the cohomology of $W_t$ changes everytime we slide over a critical point of $g$. But notice that the homology of $\mathcal{M}_\ell$ is free, with every element of $\dot{\mathcal{S}}(\ell)$ producing two generators by Theorem \ref{bettinumbers}. By Lemma \ref{subsetssmall} this means that sliding over a critical point of $g$ produces one extra summand of $\Z$ in $H^\ast(W_t;\Z)$ and two extra summands in $H^\ast(\mathcal{M}_{\ell_t};\Z)$.

Let $J\subset \{2,\ldots,n-1\}$ be such that $J$ is short with respect to $\ell$, but $J\cup\{1\}$ is long. With $K=\{2,\ldots,n-1\}-J$ we get $K\cup\{n\}$ is short, but $K\cup \{1,n\}$ is long. Let $t_0=g(q_J)$ and $\delta>0$ so that $g^{-1}([t_0-\delta,t_0+\delta])$ does not contain other critical points\footnote{By possibly changing $\ell$ slightly we can ensure that all critical points of $g$ have different value.}. Inductively we assume that (\ref{cohtypei}) holds for $\ell_{t_0+\delta}$ and we want to show that it holds for $\ell_{t_0-\delta}$.

First observe that $H^\ast_{(1)}(W_t;\Z)\cong H^\ast_{(1)}(\mathcal{M}_{\ell_t};\Z)$ for all regular values $t$. This is true because the surjective map $H^\ast(\pi_1(\mathcal{M}_{\ell_t});\Z)\to H^\ast _{(1)}(\mathcal{M}_{\ell_t};\Z)$ factors through $H^\ast_{(1)}(W_t;\Z)$ by Corollary \ref{fundgrptrns}, and $H^\ast_{(1)}(W_t;\Z)\to H^\ast_{(1)}(\mathcal{M}_{\ell_t};\Z)$ is injective by the discussion above.

Also $H^\ast(T^{n-1};\Z)\to H^\ast_{(1)}(\mathcal{M}_{\ell_t};\Z)$ factors through $H^\ast_{(1)}(W_t;\Z)$, and we have an extra $\Z$-summand corresponding to $X_K$ in $H^\ast_{(1)}(W_{t_0-\delta};\Z)$. Note that $0\not=X_K\in H^\ast_{(1)}(\mathcal{M}_{\ell_{t_0-\delta}};\Z)$, as it is part of the balanced subalgebra by Proposition \ref{thefacering}.

Therefore $H^\ast_{(1)}(\mathcal{M}_{\ell_{t_0-\delta}};\Z)$ is obtained from $H^\ast_{(1)}(\mathcal{M}_{\ell_{t_0+\delta}};\Z)$ by removing the relation $X_K$, which is either of the form $A_{K-\{i,\ldots,n-3\}}(A_{\{i,\ldots,n-3\}}\pm B_{\{i,\ldots,n-3\}})$ or $A_K$. This finishes the induction.

We now consider the case $i=n-4$ and $\{n-4,n-3,n\}$ long with respect to $\ell$. The proof is by induction on $n$. For $n=5$, there is only one case and $\mathcal{M}_\ell\cong M_2$, the orientable surface of genus 2. It is clear that the statement is correct then. So assume $n\geq 6$.

Then $\pi_1(\mathcal{M}_\ell)\cong \bar{G}_\ell$ with $\bar{G}_\ell$ as in Proposition \ref{fundgrpn-4long}. We still have the diagram
\[
 \xymatrix{H^\ast(T^{n-1};\Z)\cong\Lambda_\Z[X_1,\ldots,X_{n-1}] \ar[d]^{i^\ast} \ar[rd]^{\varphi^\ast} \\
H^\ast_{(1)}(\mathcal{M}_\ell;\Z) & H^\ast(\bar{G}_\ell;\Z) \ar[l]}
\]
We claim that
\begin{eqnarray}\label{cohn-4long}
 H^\ast(\bar{G}_\ell;\Z)&\cong & \Lambda_\Z[A_1,\ldots,A_{n-1},B_{n-4},B_{n-3}]/I,
\end{eqnarray}
where $I$ is the ideal generated by the relations
\begin{gather*}
 A_{n-4}A_{n-3}+B_{n-4}B_{n-3} \\
A_j B_k \mbox{ for }  j\in\{n-4,\ldots,n-1\}, k\in\{n-4,n-3\} \\
A_j A_k \mbox{ for }  j\in\{n-2,n-1\}\mbox{ with } \{k,j,n\}\mbox{ long,}\\
A_j  \mbox{ if }  \{j,n\} \mbox{ is long,}
\end{gather*}
and the homomorphism $\varphi^\ast:H^\ast(T^{n-1};\Z)\to H^\ast(\bar{G}_\ell;\Z)$ satisfies
\begin{eqnarray*}
\varphi^\ast(X_j)&=&A_j\,\,\,\mbox{ for } j\notin \{n-4,n-3\}\\
\varphi^\ast(X_j)&=&A_j-B_j\,\,\,\mbox{ for }j\in \{n-4,n-3\}.
\end{eqnarray*}
Using Proposition \ref{fundgrpn-4long}, this is easily seen to be true if $\{n-2,n\}$ is long, as in that case $\bar{G}_\ell\cong \Z^{n-5}\times S_2$, and we use $0=X_{\{n-4,n-3\}}=A_{n-4}A_{n-3}+B_{n-4}B_{n-3}$. In the general case, we get that $\bar{G}_\ell$ is obtained from $\Z^{n-5}\times S_2$ by up to two trivial HNN-extensions. To be more precise, if $\{n-2,n\}$ is short, we get an extra generator $a_{n-2}$ which commutes with $a_i$ provided that $\{i,n-2,n\}$ is short, and has no relations with any of the other generators. Let $k\leq n-5$ be the largest number with $\{k,n-2,n\}$ short, and $H_1$ be the corresponding HNN-extension of $\Z^{n-5}\times S_2$ along $\Z^k\subset \Z^{n-5}$. From the long exact sequence for an HNN-extension \cite[Ch.VII.9]{brown}, we get
\begin{eqnarray*}
 H^\ast(H_1;\Z)&\cong &H^\ast(\Z^{n-5}\times S_2;\Z)\oplus H^{\ast-1}(\Z^k;\Z)
\end{eqnarray*}
as abelian groups, since the $\alpha$-homomorphism in that sequence is zero as we have a trivial HNN-extension and trivial coefficients. Now $H^0(\Z^k;\Z)\subset H^1(H_1;\Z)$ is generated by $A_{n-2}$, and the other generators of $H^{\ast-1}(\Z^k;\Z)$ correspond to $A_JA_{n-2}$ with $J\subset\{1,\ldots,k\}$. There are no relations among these, as $\Z^k\times \Z$, generated by $a_1,\ldots,a_k,a_{n-2}$ is a retract of $H_1$. If $\{n-1,n\}$ is long, we get $H_1\cong \bar{G}_\ell$ and (\ref{cohn-4long}) follows.

If $\{n-1,n\}$ is also short, then $\bar{G}_\ell$ is a trivial HNN-extension of $H_1$ along a subgroup $\Z^l$ generated by $a_1,\ldots,a_l$, where $l\leq n-5$ is the largest number with $\{l,n-1,n\}$ short. Repeating the argument above gives (\ref{cohn-4long}) in this case as well.

The relations $A_J$ for $J\cap \{n-2,n-1\}\not=\emptyset$ are contained in $I_\ell$ by the same argument as in the case $i\leq n-5$. It follows that all the relations listed in the statement of the theorem are contained in $I_\ell$ and we have to show that no other relations are needed. But the proof of this is identical to the argument used in the case $i\leq n-5$ above.

It remains to consider the case $i=n-4$ and $\{n-4,n-3,n\}$ short with respect to $\ell$. As mentioned above Proposition \ref{fundgrpn-4short}, we have $n\geq 7$ in that case. Again all the relations are contained in $I_\ell$ by an argument identical to the case $i\leq n-5$. Note that the fundamental group is now again a right-angled Artin group.

The argument that no other relations are needed is similar to the arguments above, but we have to be slightly careful with the induction step. The point is that $\{n-4,n-3,n\}$ may or may not be short with respect to $\ell_{\ell_1-\varepsilon}$. If it is in long with respect to $\ell_{\ell_1-\varepsilon}$, the statement holds for this length vector by the case considered above, and the cobordism argument used in the case $i\leq n-5$ carries over and shows that the result also holds for $\ell$. Passing over the critical point corresponding to $\{n-4,n-3\}$ simply removes the relation $A_{n-4}A_{n-3}+B_{n-4}B_{n-3}$.

If $n=7$, we necessarily have that $\{n-4,n-3,n\}$ is long with respect to $\ell_{\ell_1-\varepsilon}$, so we get an induction start. Induction on $n$, using the cobordism argument once again, now gives the desired result.
\end{proof}

Define the annihilator
\begin{eqnarray*}
 {\rm Ann}(\ell)&=&\{x\in H^1(\mathcal{M}_\ell;\Z)\,|\,xy_1\cdots y_{n-4}=0\,\,\,\forall y_1,\ldots,y_{n-4}\in H^1(\mathcal{M}_\ell;\Z) \}.
\end{eqnarray*}

Similarly, for $x\in H^1(\mathcal{M}_\ell;\Z)$ define
\begin{eqnarray*}
 {\rm Ann}(x;\ell)&=&\{y\in H^1(\mathcal{M}_\ell;\Z)\,|\,xy=0\}.
\end{eqnarray*}

Notice that these annihilators can be defined for any $\ell$.

\begin{lemma}\label{annihilator}
 Let $\ell$ be special of type $\{i,n-2,n-1\}$. Then
\begin{eqnarray*}
 {\rm Ann}(\ell)&=&\langle A_{n-2},A_{n-1}\rangle.
\end{eqnarray*}

\end{lemma}

\begin{proof}
 This follows easily from Theorem \ref{cohtypein-4} after observing that $A_1\cdots A_{n-3}=\pm A_1\cdots A_{i-1}B_i\cdots B_{n-3}\not=0$. Compare also the proof of Lemma \ref{annihi}.
\end{proof}

\begin{lemma}\label{annix}
 Let $\ell$ be special of type $\{i,n-2,n-1\}$, $m\in\{1,\ldots,i-1\}$ with $i\leq n-4$ and $a_m,\ldots, a_{n-3}, b_i,\ldots,b_{n-3}\in \Z$ with $a_m\not=0$. Let
\begin{eqnarray*}
 x&=&\sum_{j=m}^{n-3}a_jA_j+\sum_{j=i}^{n-3}b_j B_j,
\end{eqnarray*}
then 
\begin{eqnarray*}
{\rm rank}\,{\rm Ann}(x;\ell)\,\,\,=\,\,\,{\rm rank}\,{\rm Ann}(A_m;\ell)&\leq & 3
\end{eqnarray*}
with equality if and only if $\{n-1,n\}$ is short and $\{m,n-2,n\}$ is long with respect to $\ell$.
\end{lemma}

\begin{proof}
First observe that ${\rm Ann}(A_m;\ell)\subset \langle A_m,A_{n-2},A_{n-1}\rangle$ by Theorem \ref{cohtypein-4}, so the rank can be at most 3, with equality if and only if $\{n-1,n\}$ is short and $\{m,n-2,n\}$ is long with respect to $\ell$.

Let
\begin{eqnarray*}
 y&=&\sum_{j=1}^{n-1}c_jA_j+\sum_{j=i}^{n-3}d_j B_j
\end{eqnarray*}
and assume $y\in {\rm Ann}(x;\ell)$. Note that if $A_mA_k=0$ for $k\in\{n-2,n-1\}$, we also get $A_jA_k=0$ for $j\geq m$, and $c_k\in\Z$ can be arbitrary. But if $A_mA_k\not=0$ for $k\in\{n-2,n-1\}$, we get $c_k=0$, since $x$ does not have an nonzero $A_k$-coefficient. Hence for $k\geq n-2$ we get $A_k\in {\rm Ann}(x;\ell)$ if and only if $A_k\in {\rm Ann}(A_m;\ell)$. We will therefore now assume $c_k=0$ for $k\in \{n-1,n-2\}$.

For $j\not=m$ the coefficient of $A_mA_j$ in $xy$ is $a_mc_j-a_jc_m$, and the coefficient of $A_mB_j$ is $a_md_j-b_jc_m$. So for $j<m$ we get $c_j=0$. Consider the matrix
\begin{eqnarray*}
M&=&\begin{pmatrix} a_{m+1} & -a_m \\ \vdots & & \ddots \\ a_{n-3} & & & -a_m \\ b_i & & & & -a_m \\ \vdots & & & & & \ddots \\ b_{n-3} & & & & & & -a_m\end{pmatrix}
\end{eqnarray*}
with zeros in the empty spaces. Then $My=0$ is equivalent to the the coefficients of $A_mA_j$ and $A_mB_j$ for $j=m+1,\ldots,n-3$ in $xy$ are zero. As the rank of $\ker M$ is clearly 1, we get the desired result.
\end{proof}

\begin{theorem}\label{cohrigin-1}
 For $i\leq n-4$, type $\{i,n-2,n-1\}$ is cohomologically rigid.
\end{theorem}

\begin{proof}
 Let $\ell$ and $\ell'$ be special length vectors of type $\{i,n-2,n-1\}$, and let $\varphi:H^\ast(\mathcal{M}_\ell;\Z)\to H^\ast(\mathcal{M}_{\ell'};\Z)$ be an isomorphism. This restricts to an isomorphism $\varphi:H^\ast_{(1)}(\mathcal{M}_\ell;\Z)\to H^\ast_{(1)}(\mathcal{M}_{\ell'};\Z)$. By Theorem \ref{cohtypein-4} we thus have an isomorphism
\begin{eqnarray*}
 \Lambda_\Z[A_1,\ldots,A_{n-1},B_i,\ldots,B_{n-3}]/I_\ell&\cong &\Lambda_\Z[A'_1,\ldots,A'_{n-1},B'_i,\ldots,B'_{n-3}]/I_{\ell'}
\end{eqnarray*}
where the `prime' always refers to $\ell'$, and the relations are given by Theorem \ref{cohtypein-4}. It follows from Lemma \ref{annihilator} that $\varphi(A_k)\in \langle A'_{n-2},A'_{n-1}\rangle$ for $k\geq n-2$. In particular, there is a unique $T_k\in \langle A_{n-2},A_{n-1}\rangle$ with $\varphi(T_k)=A'_k$ for $k=n-2,n-1$.

Here we allow the possibility that $A_k=0$ for $k\geq n-2$, but observe that for rank reasons we have $A_k=0$ if and only if $A'_k=0$. Also notice that $\varphi:H^1(\mathcal{M}_\ell;\Z)\to H^1(\mathcal{M}_{\ell'};\Z)$ induces an isomorphism
\[
 \bar{\varphi}:\Lambda_\Z[A_1,\ldots,A_{n-1},B_i,\ldots,B_{n-3}] \to \Lambda_\Z[A'_1,\ldots,A'_{n-1},B'_i,\ldots,B'_{n-3}]
\]
sending $I_\ell$ to $I_{\ell'}$.

Given $a_j,b_j\in \Z$ for $j=i,\ldots,n-3$, define
\[
 \psi:\Lambda_\Z[A_1,\ldots,A_{n-1},B_i,\ldots,B_{n-3}]\to \Lambda_\Z[A'_1,\ldots,A'_{n-1},B'_i,\ldots,B'_{n-3}]
\]
by $\psi(A_j)=\bar{\varphi}(A_j)$ for $j\in \{1,\ldots,i-1,n-2,n-1\}$, and
\begin{eqnarray*}
\psi(A_j)&=&\bar{\varphi}(A_j+a_jT_{n-1}) \mbox{ for }j=i,\ldots,n-3 \\
\psi(B_j)&=&\bar{\varphi}(B_j+b_jT_{n-1}) \mbox{ for }j=i,\ldots,n-3.
\end{eqnarray*}
Using $A_kA_j=A_kB_j=0$ for $k\geq n-2$ and $j\geq i$, it is easy to see that $\psi$ sends $I_\ell$ to $I_{\ell'}$, so we get a surjective map $\psi:H^\ast_{(1)}(\mathcal{M}_\ell;\Z) \to H^\ast_{(1)}(\mathcal{M}_{\ell'};\Z)$. Since both have the same rank, it is in fact an isomorphism. Notice that $\bar{\varphi}(A_j+a_jT_{n-1})=\bar{\varphi}(A_j)+a_j A'_{n-1}$, so we can choose the $a_j$ and $b_j$ with
\begin{eqnarray*}
 \psi(\langle A_i,\ldots,A_{n-3},B_i,\ldots,B_{n-3}\rangle) &\subset & \langle A'_1,\ldots,A'_{n-2},B'_i,\ldots,B'_{n-3}\rangle.
\end{eqnarray*}
Repeating the argument with $T_{n-2}$, we can assume that
\begin{eqnarray*}
 \bar{\varphi}(\langle A_i,\ldots,A_{n-3},B_i,\ldots,B_{n-3}\rangle) &\subset & \langle A'_1,\ldots,A'_{n-3},B'_i,\ldots,B'_{n-3}\rangle.
\end{eqnarray*}
Now observe that for $j\in \{i,\ldots,n-3\}$ we get
\begin{eqnarray*}
\langle A_j, A_{n-2}, A_{n-1}, B_i,\ldots, B_{n-3} \rangle & \subset & {\rm Ann}(A_j;\ell)\\
\langle B_j, A_i,\ldots, A_{n-1}\rangle & \subset & {\rm Ann}(B_j;\ell),
\end{eqnarray*}
so the ranks are at least 3, and the rank equal to 3 implies $A_{n-2}=0$. By Lemma \ref{annix} we therefore get
\begin{eqnarray*}
 \bar{\varphi}(\langle A_i,\ldots,A_{n-3},B_i,\ldots,B_{n-3}\rangle) &\subset & \langle A'_i,\ldots,A'_{n-3},B'_i,\ldots,B'_{n-3}\rangle.
\end{eqnarray*}
If $J\subset \{1,\ldots,n-1\}$ with $J=\{j_1,\ldots,j_s\}$ satisfying
\[
j_1<\ldots<j_l\leq i-1<j_{l+1}<\ldots <j_r\leq n-3 <j_{r+1}<\ldots <j_s,
\]
let us write
\begin{eqnarray*}
 (AB)_J&=&A_{j_1}\cdots A_{j_l} B_{j_{l+1}}\cdots B_{j_r} A_{j_{r+1}}\cdots A_{j_s}.
\end{eqnarray*}
Note that if $J\cap \{i,\ldots,n-3\}=\emptyset$, we simply have $(AB)_J=A_J$.

Let $K_\ell$ be the ideal of $\Lambda_\Z[A_1,\ldots,A_{n-1},B_i,\ldots,B_{n-3}]$ generated by
\begin{eqnarray*}
 A_J&\mbox{for}& J\cup\{n\} \mbox{ long with respect to }\ell\\
(AB)_J&\mbox{for}& J\cup\{n\} \mbox{ long with respect to }\ell\\
A_j B_k&\mbox{for}& j,k\in\{i,\ldots, n-3\}.
\end{eqnarray*}
Notice that this includes the relations
\begin{eqnarray*}
A_j B_k&\mbox{for}& j=n-2,n-1,\, i\leq k\leq n-3,
\end{eqnarray*}
so $I_\ell\subset K_\ell$. Similarly define $K_{\ell'}$ as an ideal of $\Lambda_\Z[A'_1,\ldots,A'_{n-1},B'_i,\ldots,B'_{n-3}]$ with respect to $\ell'$. Recall the isomorphism
\[
 \bar{\varphi}:\Lambda_\Z[A_1,\ldots,A_{n-1},B_i,\ldots,B_{n-3}] \to \Lambda_\Z[A'_1,\ldots,A'_{n-1},B'_i,\ldots,B'_{n-3}].
\]
We claim that $\bar{\varphi}$ induces an isomorphism
\[
 \bar{\psi}:\Lambda_\Z[A_1,\ldots,A_{n-1},B_i,\ldots,B_{n-3}]/K_\ell \to \Lambda_\Z[A'_1,\ldots,A'_{n-1},B'_i,\ldots,B'_{n-3}]/K_{\ell'}.
\]
First we need to show that this induces a well defined ring homomorphism. To do this, it is enough to show that $\bar{\varphi}$ sends $K_\ell$ to $K_{\ell'}$. We already know that $\bar{\varphi}$ sends $I_\ell$ to $I_{\ell'}$, so we only have to consider the relations $A_J$ and $(AB)_J$ for $J\subset \{1,\ldots,n-3\}$ with $J\cup\{n\}$ long. Then $\{i,\ldots,n-3\}\subset J$ by the fact that $\ell$ is of type $\{i,n-2,n-1\}$. Let $J'=J-\{i,\ldots,n-3\}$, then
\begin{eqnarray*}
 \bar{\varphi}(A_{J'}(A_{\{i,\ldots,n-3\}}+(-1)^{n-2-i} B_{\{i,\ldots,n-3\}}))&\in& I_{\ell'}.
\end{eqnarray*}
But since $\bar{\varphi}$ sends $\langle A_i,\ldots,A_{n-3},B_i,\ldots,B_{n-3}\rangle$ to $\langle A'_i,\ldots,A'_{n-3},B'_i,\ldots,B'_{n-3}\rangle$, we get
\begin{eqnarray*}
 \bar{\varphi}(A_{\{i,\ldots,n-3\}}), \bar{\varphi}(B_{\{i,\ldots,n-3\}}) &\in& \langle A'_{\{i,\ldots,n-3\}},B'_{\{i,\ldots,n-3\}} \rangle,
\end{eqnarray*}
as the subgroup
\[
 H^{n-2-i}_{(1)}(\mathcal{M}_{\ell'};\Z)\cap \langle A'_i,\ldots,A'_{n-3},B'_i,\ldots,B'_{n-3}\rangle
\]
contains $\bar{\varphi}(A_{\{i,\ldots,n-3\}}), \bar{\varphi}(B_{\{i,\ldots,n-3\}})$, and is generated by $\langle A'_{\{i,\ldots,n-3\}},B'_{\{i,\ldots,n-3\}} \rangle$.

Now $\bar{\varphi}(A_{J'}) \in \langle A'_1,\ldots,A'_{n-1},B'_i,\ldots,B'_{n-3} \rangle$, but any summands involving factors of $A'_j$ or $B'_j$ with $j\geq i$ end up in $K_{\ell'}$ after multiplication with $A'_{\{i,\ldots,n-3\}}$ or $B'_{\{i,\ldots,n-3\}}$. So modulo $K_{\ell'}$, we get
\begin{eqnarray*}
 \bar{\varphi}(A_{J'})\cdot \bar{\varphi}(A_{\{i,\ldots,n-3\}}) & = & \sum_{\tilde{J}\subset \{1,\ldots,i-1\}}a_{\tilde{J}}A'_{\tilde{J}}\cdot \bar{\varphi}(A_{\{i,\ldots,n-3\}})\\
\bar{\varphi}(A_{J'})\cdot \bar{\varphi}((-1)^{n-2-i} B_{\{i,\ldots,n-3\}}) & = & \sum_{\tilde{J}\subset \{1,\ldots,i-1\}}a_{\tilde{J}}A'_{\tilde{J}}\cdot \bar{\varphi}((-1)^{n-2-i} B_{\{i,\ldots,n-3\}})
\end{eqnarray*}
As the sum of these equations is in $I_{\ell'}$, we have to have $a_{\tilde{J}}=0$ for all $\tilde{J}$ with $\tilde{J}\cup \{i,\ldots,n-3,n\}$ short with respect to $\ell'$. But then both $\bar{\varphi}(A_{J'})\cdot \bar{\varphi}(A_{\{i,\ldots,n-3\}})$ and $\bar{\varphi}(A_{J'})\cdot \bar{\varphi}(B_{\{i,\ldots,n-3\}})$ are contained in $K_{\ell'}$.

This shows that $\bar{\varphi}$ sends $K_\ell$ to $K_{\ell'}$. Repeating the argument with $\bar{\varphi}^{-1}$ we get a homomorphism
\[
 \bar{\psi}^\ast:\Lambda_\Z[A'_1,\ldots,A'_{n-1},B'_i,\ldots,B'_{n-3}]/K_{\ell'} \to \Lambda_\Z[A_1,\ldots,A_{n-1},B_i,\ldots,B_{n-3}]/K_\ell.
\]
with $\bar{\psi}^\ast\circ \bar{\psi}={\rm id}$ and $\bar{\psi}\circ \bar{\psi}^\ast={\rm id}$ on degree 1 elements. As both rings are generated by degree 1 elements, we get that $\bar{\psi}$ and $\bar{\psi}^\ast$ are mutually inverse isomorphisms. Also $\Lambda_\Z[A_1,\ldots,A_{n-1},B_i,\ldots,B_{n-3}]/K_\ell$ and $\Lambda_\Z[A'_1,\ldots,A'_{n-1},B'_i,\ldots,B'_{n-3}]/K_{\ell'}$ are exterior face rings, so there is a bijection between $\{A_1,\ldots,A_{n-1},B_i$ $,\ldots,B_{n-3}\}$ and $\{A'_1,\ldots,A'_{n-1},B'_i,\ldots,B'_{n-3}\}$ inducing the isomorphism by Theorem \ref{isoprobsol}. But a simple argument using the annihilators ${\rm Ann}(A_j;\ell)$ shows that this bijection restricts to bijections between
\begin{eqnarray*}
\{A_1,\ldots, A_{i-1}\} &\mbox{and}& \{A'_1,\ldots,A'_{i-1}\}\\
\{A_i,\ldots,A_{n-3},B_i,\ldots,B_{n-3}\} &\mbox{and} &\{A'_i,\ldots,A'_{n-3},B'_i,\ldots,B'_{n-3}\}\\
\{A_{n-2},A_{n-1}\} &\mbox{and}& \{A'_{n-2},A'_{n-1}\}.
\end{eqnarray*}
Since $A_jB_k=0=A'_jB'_k$ for all $j,k\in \{i,\ldots,n-3\}$, we also get a bijection between $\{A_i,\ldots,A_{n-3}\}$ and $\{A'_i,\ldots, A'_{n-3}\}$ or $\{B'_i,\ldots,B'_{n-3}\}$. Since we can swap these sets, we can assume that the isomorphism is induced by a bijection between $\{A_1,\ldots, A_{n-1}\}$ and $\{A'_1,\ldots, A'_{n-1}\}$, and by a bijection between $\{B_i,\ldots, B_{n-3}\}$ and $\{B'_i,\ldots, B'_{n-3}\}$. Because of the relations in $K_\ell$ and $K_{\ell'}$, this shows that we have a bijection between $\dot{\mathcal{S}}(\ell)$ and $\dot{\mathcal{S}}(\ell')$. By \cite[Lm.3]{fahasc} we get that $\ell$ and $\ell'$ are in the same chamber. This finishes the proof.
\end{proof}

\section{Cohomology of special length vectors of type $\{n-4,n-3,n-1\}$}\label{secn-4n-3}

We now want to show that type $\{n-4,n-3,n-1\}$ is also cohomologically rigid. We start by determining the fundamental group for such special length vectors. As in the case of special length vectors of type $\{i,n-2,n-1\}$ we want to begin by constructing certain elements of the fundamental group.

First we write $\pi_1(T^{n-1})\cong \Z^{n-1}$ as generated by $x_1,\ldots,x_{n-1}$, where each $x_i$ is represented by a rotation along the $i$-th coordinate, while keeping the other coordinates fixed.

Define a basepoint $z^\ast=(z_1^\ast,\ldots,z_{n-1}^\ast)\in \mathcal{M}_\ell$ by $z_1^\ast,\ldots,z_{n-5}^\ast,z_{n-2}^\ast=1$, $z^\ast_{n-1}=-1$ and choose $z_{n-4}^\ast,z_{n-3}^\ast\in S^1$ so that we have a closed linkage.

As in Section \ref{seci=1}, define $\gamma_j:(S^1,1)\to (\mathcal{M}_\ell,z^\ast)$ for $j=1,\ldots,n-5$ so that $i_\#[\gamma_j]=x_j\in \pi_1(T^{n-1})$. Also, if $\{n-1,n\}$ is short, we can define $\gamma_{n-1}:(S^1,1)\to (\mathcal{M}_\ell,z^\ast)$ with $i_\#[\gamma_{n-1}]=x_{n-1}$. It is also easy to see that $[\gamma_j],[\gamma_k]$ commute for $j,k\leq n-5$, and $[\gamma_{n-1}]$ commutes with $[\gamma_j]$, provided that $\{j,n-1,n\}$ is short.

\begin{proposition}\label{fundgrpn-3n-1}
 Let $n\geq 5$ and $\ell$ a special length vector of type $\{n-4,n-3,n-1\}$. Then the fundamental group of $\mathcal{M}_\ell$ has a presentation with generating set
\[
 \{a_1,\ldots,a_{n-1},b_{n-4},b_{n-3},b_{n-2}\}
\]
subject to the relations
\begin{eqnarray*}
 [a_{n-4},b_{n-4}][a_{n-3},b_{n-3}][a_{n-2},b_{n-2}]& & \\
\,[a_j,a_k] & \mbox{if} & \{j,k,n\} \mbox{ is short}\\
\,[a_j,b_k] & \mbox{if} & j\leq n-5, k\in \{n-4,n-3,n-2\}\\
a_{n-1} & \mbox{if} & \{n-1,n\}\mbox{ is long}.
\end{eqnarray*}
Furthermore, the homomorphism $i_\#:\pi_1(\mathcal{M}_\ell)\to\pi_1(T^{n-1})$ induced by inclusion sends $a_j$ to $x_j$ for $j\in\{1,\ldots,n-5\}$ and also for $j=n-1$, provided that $\{n-1,n\}$ is short. For $j\in\{n-4,n-3,n-2\}$ we also get that $i_\#(a_j),i_\#(b_j)$ is contained in the subgroup generated by $x_{n-4},x_{n-3},x_{n-2}$.
\end{proposition}

\begin{remark}
 One could ask for a sharpening of Proposition \ref{fundgrpn-3n-1} by determining the image of $a_j$ and $b_j$ for $j\in\{n-4,n-3,n-2\}$ by giving precise generators. However, we do not need this information below.
\end{remark}

\begin{proof}[Proof of Proposition \ref{fundgrpn-3n-1}]
 The proof is by induction on $n$. For $n=5$, a Betti number argument shows that $\mathcal{M}_\ell\cong M_3$, the orientable surface of genus 3. Furthermore, $\{4,5\}$ is long, so the statement of the proposition holds.

For $n>5$, note that $\ell'=(\ell_2,\ldots,\ell_{n-1},\ell_n+\ell_1)$ is of type $\{n-5,n-4,n-2\}$, and we get a cobordism $W$ from Lemma \ref{critpointsf} between $\mathcal{M}_\ell$ and $S^1\times \mathcal{M}_{\ell'}$, and the proposition holds for $\bar{\ell}=(\varepsilon,\ell_2,\ldots,\ell_{n-1},\ell_n+\ell_1-\varepsilon)$. As in the proof of Theorem \ref{fundgrpin-1}, we can get one extra generator, which is then represented by $\gamma_{n-1}$ above. This shows that the statement about $i_\#:\pi_1(\mathcal{M}_\ell)\to \pi_1(T^{n-1})$ holds.

Furthermore, for every two element subset $\{i,j\}\subset \{2,\ldots,n-1\}$ with $\{i,j,n\}$ short with respect to $\ell$ and long with respect to $\bar{\ell}$ we get a commutator $[a_i,a_j]$. Notice that since $\ell$ and $\bar{\ell}$ are of type $\{n-4,n-3,n-1\}$, we necessarily get $i\leq n-5$ and $j=n-1$ in this situation.
\end{proof}

\begin{remark}
 If $\{n-1,n\}$ is long, the fundamental group is simply $\Z^{n-5}\times S_3$, where $S_3=\pi_1(M_3)$. In fact, it is easy to see that there is only one chamber of type $\{n-4,n-3,n-1\}$ up to permutation with the additional property that $\{n-1,n\}$ is long, compare Lemma \ref{onelowtype}. More precisely, one shows that for this chamber we get $\mathcal{M}_\ell\cong T^{n-5}\times M_3$ by an argument similar to Corollary \ref{thelowtype}.
\end{remark}

For this reason, we can assume that $\{n-1,n\}$ is short with respect to $\ell$. In this case we get that the fundamental group of $\mathcal{M}_\ell$ is a trivial HNN-extension of $Z^{n-5}\times S_3$ along the subgroup $\Z^k$ generated by $a_1,\ldots,a_k$, where $k$ is the maximal number lesser equal to $n-5$ such that $\{k,n-1,n\}$ is short. In case there is no such $k$ (that is, $k=0$), we simply get the free product of $\Z^{n-5}\times S_3$ with $\Z$.

\begin{corollary}\label{cohfundgrp}
 Let $\ell$ be special of type $\{n-4,n-3,n-1\}$ with $\{n-1,n\}$ short and $n\geq 5$. Then
\begin{eqnarray*}
 H^\ast(\pi_1(\mathcal{M}_\ell);\Z)&\cong & \Lambda_\Z[A_1,\ldots,A_{n-1},B_{n-4},B_{n-3},B_{n-2}]/I
\end{eqnarray*}
where $I$ is the ideal generated by
\begin{eqnarray*}
 A_iA_j & \mbox{for} & i,j\in \{n-4,n-3,n-2\} \\
B_iB_j&\mbox{for}& i,j\in\{n-4,n-3,n-2\}\\
A_iB_j&\mbox{for}&i,j\in\{n-4,n-3,n-2\}\mbox{ with }i\not=j\\
A_{n-1}A_j & \mbox{for} & j\leq n-2 \mbox{ with }\{j,n-1,n\} \mbox{ long}\\
A_{n-1}B_j&\mbox{for}&j\in\{n-4,n-3,n-2\}\\
A_{n-2}B_{n-2}-A_{n-3}B_{n-3}&\mbox{and}&A_{n-3}B_{n-3}-A_{n-4}B_{n-4}
\end{eqnarray*}

\end{corollary}

\begin{proof}
 From the long exact sequence for an HNN-extension \cite[Ch.VII.9]{brown}, we get
\begin{eqnarray*}
 H^\ast(\pi_1(\mathcal{M}_\ell);\Z)&\cong &H^\ast(\Z^{n-5}\times S_3;\Z)\oplus H^{\ast-1}(\Z^k;\Z)
\end{eqnarray*}
as abelian groups, since the $\alpha$-homomorphism in that sequence is zero as we have a trivial HNN-extension and trivial coefficients. Now $H^0(\Z^k;\Z)\subset H^1(\pi_1(\mathcal{M}_\ell);\Z)$ is generated by $A_{n-1}$, and the other generators of $H^{\ast-1}(\Z^k;\Z)$ correspond to $A_JA_{n-1}$ with $J\subset\{1,\ldots,k\}$. There are no relations among these, as $\Z^k\times \Z$, generated by $a_1,\ldots,a_k,a_{n-1}$ is a retract of $\pi_1(\mathcal{M}_\ell)$.

The relations for $I$ that do not involve $A_{n-1}$ hold, as they hold in 
\begin{eqnarray*}
H^\ast(\Z^{n-5}\times S_3;\Z)&\cong& H^\ast(T^{n-5};\Z)\otimes H^\ast(M_3;\Z),
\end{eqnarray*}
and the relations involving $A_{n-1}$ hold, as $(\Z^{n-5-k}\times S_3)\ast \Z$ is a retract of $\pi_1(\mathcal{M}_\ell)$. This finishes the proof.
\end{proof}

\begin{proposition}\label{cohtypen-3n-1}
 Let $\ell$ be special of type $\{n-4,n-3,n-1\}$ with $n\geq 5$. Then
\begin{eqnarray*}
 H^\ast_{(1)}(\mathcal{M}_\ell;\Z)&\cong & \Lambda_\Z[A_1,\ldots,A_{n-1},B_{n-4},B_{n-3},B_{n-2}]/I_\ell
\end{eqnarray*}
where $I_\ell$ is the ideal generated by
\begin{eqnarray*}
A_J&\mbox{for} & J\cup\{n\}\mbox{ long with } J\cap\{n-4,n-3,n-2\}=\emptyset\\
 A_iA_j & \mbox{for} & i,j\in \{n-4,n-3,n-2\} \\
B_iB_j&\mbox{for}& i,j\in\{n-4,n-3,n-2\}\\
A_iB_j&\mbox{for}&i,j\in\{n-4,n-3,n-2\}\mbox{ with }i\not=j\\
A_{n-1}A_j & \mbox{for} & j\leq n-2 \mbox{ with }\{j,n-1,n\} \mbox{ long}\\
A_{n-1}B_j&\mbox{for}&j\in\{n-4,n-3,n-2\}\\
A_{n-2}B_{n-2}-A_{n-3}B_{n-3}&\mbox{and}&A_{n-3}B_{n-3}-A_{n-4}B_{n-4}
\end{eqnarray*}
\end{proposition}

\begin{proof}
 The proof is similar to the proof of Theorem \ref{cohtypein-4}, and by induction over $n$. For $n=5$, we get $\mathcal{M}_\ell\cong M_3$ and the result is obvious.

For $n>5$ we have the diagram
\[
 \xymatrix{H^\ast(T^{n-1};\Z)\cong\Lambda_\Z[X_1,\ldots,X_{n-1}] \ar[d]^{i^\ast} \ar[rd]^{\varphi^\ast} \\
H^\ast_{(1)}(\mathcal{M}_\ell;\Z) & H^\ast(\pi_1(\mathcal{M}_\ell);\Z) \ar[l]}
\]
Furthermore, the map $H^\ast(\pi_1(\mathcal{M}_\ell);\Z)\to H^\ast_{(1)}(\mathcal{M}_\ell;\Z)$ induced by the classifying map is surjective and an isomorphism for $\ast\leq 1$. Therefore
\begin{eqnarray*}
 H^\ast_{(1)}(\mathcal{M}_\ell;\Z)&\cong & \Lambda_\Z[A_1,\ldots,A_{n-1},B_{n-4},B_{n-3},B_{n-2}]/I_\ell
\end{eqnarray*}
for some ideal $I_\ell$, and the relations listed in the statement, except the $A_J$ relations, are in $I_\ell$ by Corollary \ref{cohfundgrp}. To see that $A_J$ for $J\cup\{n\}$ long with $J\cap\{n-4,n-3,n-2\}=\emptyset$ also holds, note that by the statement on $i_\#$ in Proposition \ref{fundgrpn-3n-1} we get $i^\ast(X_j)=A_j$ for $j\in \{1,\ldots,n-5\}$, and also for $j=n-1$ if $\{n-1,n\}$ is short. Hence the relation $A_J$ holds by Proposition \ref{thefacering}.

We therefore have to show that no other relations hold. Let $W$ be the cobordism from Lemma \ref{critpointsf} between $\mathcal{M}_\ell$ and $\mathcal{M}_{\ell_{\ell_1-\varepsilon}}$. Since $\mathcal{M}_{\ell_{\ell_1-\varepsilon}}\cong S^1\times \mathcal{M}_{\ell'}$, the result holds for $\ell_{\ell_1-\varepsilon}$. As in the proof of Theorem \ref{cohtypein-4}, whenever we slide over a critical point (going from $\mathcal{M}_{\ell_{\ell_1-\varepsilon}}$ to $\mathcal{M}_\ell$) we will remove one relation. But note that the critical points of $g$ correspond to subsets $J\subset\{2,\ldots,n-1\}$ which are short with respect to $\ell$, but with $\{1\}\cup J$ long with respect to $\ell$. This means that $\tilde{J}=\{2,\ldots,n-1\}-J$ satisfies $\tilde{J}\cup\{n\}$ is short with respect to $\ell$ while it is long with respect to $\ell_{\ell_1-\varepsilon}$. It follows that $\tilde{J}\cap \{n-2,n-3,n-4\}=\emptyset$ and $\{n-1\}\in \tilde{J}$, for otherwise we have $\tilde{J}\cup \{n\}$ short or long with respect to both $\ell$ and $\ell_{\ell_1-\varepsilon}$ (recall they both have the same type). In particular, $A_{\tilde{J}}$ is zero in the cohomology of $\ell_{\ell_1-\varepsilon}$, but not in the cohomology of $\ell$.
\end{proof}

\begin{theorem}\label{cohrign-3n-1}
 Type $\{n-4,n-3,n-1\}$ is cohomologically rigid.
\end{theorem}

Before we start the proof, recall the annihilators defined in Section \ref{sectypen-4}. The following Lemma is proven analogously to Lemmata \ref{annihilator} and \ref{annix}.

\begin{lemma}\label{annitypen-3n-1}
 Let $\ell$ be a special length vector of type $\{n-4,n-3,n-1\}$ with $\{n-1,n\}$ short. Then
\begin{eqnarray*}
 {\rm Ann}(\ell)&=&\langle A_{n-1}\rangle.
\end{eqnarray*}
Furthermore, for $m\in\{1,\ldots,n-5\}$ and $a_m,\ldots, a_{n-2}, b_{n-4},\ldots,b_{n-2}\in \Z$ with $a_m\not=0$, let
\begin{eqnarray*}
 x&=&\sum_{j=m}^{n-2}a_jA_j+\sum_{j=n-4}^{n-2}b_j B_j,
\end{eqnarray*}
then
\begin{eqnarray*}
{\rm rank}\,{\rm Ann}(x;\ell)\,\,\,=\,\,\,{\rm rank}\,{\rm Ann}(A_m;\ell)&=&\left\{\begin{array}{cl}1& A_mA_{n-1}\not=0 \\ 2 & A_mA_{n-1}=0 \end{array} \right.
\end{eqnarray*}
\end{lemma}

\begin{proof}[Proof of Theorem \ref{cohrign-3n-1}]
 Let $\ell$, $\ell'$ be two ordered special length vectors of type $\{n-4,n-3,n-1\}$ with isomorphic cohomology rings. Then there is an isomorphism
\[
 \varphi:\Lambda_\Z[A_1,\ldots,A_{n-1},B_{n-4},\ldots,B_{n-2}]/I_\ell \to \Lambda_\Z[A'_1,\ldots,A'_{n-1},B'_{n-4},\ldots,B'_{n-2}]/I_{\ell'}
\]
with $I_\ell$, $I_{\ell'}$ as in Proposition \ref{cohtypen-3n-1}. As $\varphi$ sends ${\rm Ann}(\ell)$ to ${\rm Ann}(\ell')$ we get $\varphi(A_{n-1})=\pm A'_{n-1}$ by Lemma \ref{annitypen-3n-1}. Also, by the same argument as in the proof of Theorem \ref{cohrigin-1}, and using Lemma \ref{annitypen-3n-1}, we can assume that
\begin{eqnarray*}
 \varphi(\langle A_{n-4},\ldots,A_{n-2},B_{n-4},\ldots,B_{n-2}\rangle)&\subset & \langle A'_{n-4},\ldots,A'_{n-2},B'_{n-4},\ldots,B'_{n-2}\rangle.
\end{eqnarray*}
The same also holds for $\varphi^{-1}$, and we get an isomorphism
\[
 \psi:\Lambda_\Z[A_1,\ldots,A_{n-5},A_{n-1}]/\tilde{I}_\ell \to \Lambda_\Z[A'_1,\ldots,A'_{n-5},A'_{n-1}]/\tilde{I}_{\ell'}
\]
where $\tilde{I}_\ell$ is the ideal generated by $A_J$, where $J\subset\{1,\ldots,n-5,n-1\}$ satisfies $J\cup\{n\}$ is long with respect to $\ell$, and $\tilde{I}_{\ell'}$ is the analogous ideal for $\ell'$.

Denote by $\tilde{\mathcal{S}}^{(n-1)}(\ell)$ the full subcomplex of $\tilde{\mathcal{S}}(\ell)$ containing $\{i\}$ for $i\in\{1,\ldots,n-5,n-1\}$, that is, it consists of those subsets $J\in \dot{\mathcal{S}}(\ell)$ for which $J\cap \{1,\ldots,n-5,n-1\}\not=\emptyset$, but $J\cap \{n-4,n-3,n-2\}=\emptyset$. Observe that $\psi$ is an isomorphism between exterior face rings
\[
 \psi:\Lambda_\Z[\tilde{\mathcal{S}}^{(n-1)}(\ell)]\to \Lambda_\Z[\tilde{\mathcal{S}}^{(n-1)}(\ell')]
\]
so by Theorem \ref{isoprobsol} we get a permutation $\sigma:\{1,\ldots,n-5,n-1\}\to \{1,\ldots,n-5,n-1\}$ inducing a bijection $\tilde{\mathcal{S}}^{(n-1)}(\ell)\to\tilde{\mathcal{S}}^{(n-1)}(\ell')$. By \cite[Lm.3]{fahasc} we\footnote{Lemma 3 of \cite{fahasc} is not stated for $\tilde{\mathcal{S}}^{(n-1)}(\ell)$, but the proof works identically for this set as well.} get $\tilde{\mathcal{S}}^{(n-1)}(\ell)=\tilde{\mathcal{S}}^{(n-1)}(\ell')$. The set $\dot{\mathcal{S}}(\ell)-\tilde{\mathcal{S}}^{(n-1)}(\ell)$ is completely determined by $\ell$ being of type $\{n-4,n-3,n-1\}$, so $\dot{\mathcal{S}}(\ell)=\dot{\mathcal{S}}(\ell')$, which means that $\ell$ and $\ell'$ are in the same chamber.
\end{proof}

\begin{proof}[Proof of Walker's Conjecture]
 Let $\ell$ and $\ell'$ be generic length vectors with isomorphic graded cohomology rings $H^\ast(\mathcal{M}_\ell;\Z)\cong H^\ast(\mathcal{M}_{\ell'};\Z)$. By Theorem \ref{normalwalker} we get that they are in the same chamber up to permutation, if one of them is normal. The same is true if the polygon spaces are disconnected. Therefore we can assume that both $\ell$ and $\ell'$ are special. By Corollary \ref{sametype} they have the same type. By Remark \ref{alltypes}, Theorems \ref{rigidn-3}, \ref{cohrigin-1} and \ref{cohrign-3n-1}, and Lemma \ref{onelowtype} we get that $\ell$ and $\ell'$ are in the same chamber up to permutation.
\end{proof}

\begin{remark}
 Note that we have used Theorem \ref{isoprobsol} with $R=\Z$ in all cases except in Theorem \ref{rigidn-3}, where we used $R=\Q$. The case $R=\Z$ can be reduced to $R=\Z/2$, so one may ask if we can obtain cohomology with $\Z$-coefficients also for type $\{n-3,n-2,n-1\}$ and use Theorem \ref{isoprobsol} with $R=\Z/2$ throughout.

Let us indicate that this is indeed possible. For $n=4$ with $\{1,2,3\}$ long, it is easy to see that $\mathcal{M}_\ell\cong S^1$ and the homomorphism $i_\#:\pi_1(\mathcal{M}_\ell)\to \pi_1(T^4)$ is trivial. For $n\geq 5$, we can always keep the first $n-4$ bars fixed at 1 to produce an element $b\in \pi_1(\mathcal{M}_\ell)$ with $i_\# b=0\in \pi_1(T^{n-1})$. Using the techniques above for calculating the fundamental group, it is now easy to see that the fundamental group is a right-angled Artin group
\begin{eqnarray*}
 \pi_1(\mathcal{M}_\ell)&=&\left\langle a_1,\ldots,a_{n-1},b\,\left|\,
\begin{array}{rl}
a_k & \mbox{ for }\{k,n\}\mbox{ long,}\\
\mbox{ }[a_i,b]& \mbox{ for }i\leq n-4, \\
\mbox{ }[a_i,a_j] & \mbox{ for }\{i,j,n\}\mbox{ short }
\end{array}
\right. \right\rangle,
\end{eqnarray*}
and the $a_i$ are mapped to $x_i$ under $i_\#$, provided that $\{i,n\}$ is short. As before, we get for the cohomology
\begin{eqnarray*}
 H^\ast_{(1)}(\mathcal{M}_\ell;\Z)&\cong &\Lambda_\Z[A_1,\ldots,A_{n-1},B]/I_\ell
\end{eqnarray*}
where $I_\ell$ is the ideal generated by $A_J$ for $J\cup \{n\}$ long, and $A_jB$ for $j\geq n-3$. Notice that since $i_\# b=0\in \pi_1(T^{n-1})$, we get an exterior face ring. Theorem \ref{isoprobsol} can now be applied as in Theorem \ref{rigidn-3}, but with $\Z/2$-coefficients.
\end{remark}

\end{document}